\documentclass[12pt]{article}

\usepackage[utf8]{inputenc}

\usepackage[T1]{fontenc}

\usepackage[a4paper, margin=2.7cm]{geometry}

\usepackage{amsmath, amsthm, amsfonts, amssymb}
\usepackage{mathrsfs}           

\usepackage[colorlinks=true,linkcolor=blue,citecolor=blue,urlcolor=blue,breaklinks]{hyperref}

\usepackage{bbm}

\def\LL{{\mathcal L}}

\def\PP{{\mathcal P}}

\def\TT{{\mathcal T}}
\def\UU{{\mathcal U}}

\def\M{{\mathbb M}}

\DeclareMathOperator{\Jac}{Jac}

\let\oldmarginpar\marginpar
\renewcommand\marginpar[1]{\-\oldmarginpar[\raggedleft\footnotesize #1]%
{\raggedright\footnotesize #1}}

\newcommand{\beqn}{\begin{equation}}
\newcommand{\eeqn}{\end{equation}}
\newcommand{\bear}{\begin{eqnarray}}
\newcommand{\eear}{\end{eqnarray}}
\newcommand{\bean}{\begin{eqnarray*}}
\newcommand{\eean}{\end{eqnarray*}}

\def\M{\mathcal{M}}

\def\P{\mathcal{P}}
\def\R{\mathbb{R}}
\def\S{\mathbb{S}}
\def\T{\mathbb{T}}

\def\d{\,\mathrm{d}}

\def\p{\partial}

\def\isd{\int_{\S^{d-1}}}
\def\ird{\int_{\R^d}}

\def\1{\mathbbm{1}}

\def\:{\colon}

\newtheorem{thm}{Theorem}[section]

\newtheorem{lem}[thm]{Lemma}

\newtheorem{hyp}{Hypothesis}

\theoremstyle{definition}

\theoremstyle{remark}
\newtheorem*{remark}{Remark}
\newtheorem{rem}[thm]{Remark}
\theoremstyle{example}

\title{Hypocoercivity of linear kinetic equations via Harris's Theorem}
\author{José A. Cañizo\footnote{\textsc{Departamento de
    Matemática Aplicada, Universidad de Granada, 18071 Granada,
    Spain.}
  \textit{Email address}: \texttt{\href{mailto:canizo@ugr.es}{canizo@ugr.es}}} \and Chuqi Cao\footnote{ \textsc{Chuqi Cao, CEREMADE, Universit\'{e} Paris Dauphine, Place du Marechal de Lattre de Tassigny, Paris, 75775}
  \textit{E-mail address}: \texttt{\href{mailto:cao@ceremade.dauphine.fr}{cao@ceremade.dauphine.fr}}}
\and Josephine Evans\footnote{\textsc{Josephine Evans, CEREMADE, Universit\'{e} Paris Dauphine, Place du Marechal de Lattre de Tassigny, Paris, 75775}
  \textit{E-mail address}: \texttt{\href{mailto:josephine.evans@dauphine.psl.eu}{josephine.evans@dauphine.psl.eu}}} \and Havva Yolda\c{s} \footnote{\textsc{Havva Yoldaş. Basque Center for Applied Mathematics, Alameda de Mazarredo 14,
    48009 Bilbao, Spain \& Departamento de
    Matemática Aplicada, Universidad de Granada, 18071 Granada,
    Spain.}
  \textit{Email address}: \texttt{\href{mailto:hyoldas@bcamath.org}{hyoldas@bcamath.org}}}
}

\begin{document}

\maketitle

\begin{abstract}
  We study convergence to equilibrium of the linear relaxation
  Boltzmann (also known as linear BGK) and the linear Boltzmann
  equations either on the torus
  $(x,v) \in \mathbb{T}^d \times \mathbb{R}^d$ or on the whole space
  $(x,v) \in \mathbb{R}^d \times \mathbb{R}^d$ with a confining
  potential. We present explicit convergence results in total
  variation or weighted total variation norms (alternatively $L^1$ or
  weighted $L^1$ norms). The convergence rates are exponential when
  the equations are posed on the torus, or with a confining potential
  growing at least quadratically at infinity. Moreover, we give
  algebraic convergence rates when subquadratic potentials
  considered. We use a method from the theory of Markov processes
  known as Harris's Theorem.
\end{abstract}

\tableofcontents

\section{Introduction}

The goal of this paper is to give quantitative rates of convergence to
equilibrium for some linear kinetic equations, using a method based on
Harris's Theorem from the theory of Markov processes \cite{H56, MT93,
  HM11} that we believe is very well adapted to hypocoercive, nonlocal
equations. We consider equations of the type
\begin{equation*}
  \partial_t f + v \cdot \nabla_x f
  = \mathcal{L} f,
\end{equation*}
where $f = f(t, x, v)$, with time $t \geq 0$, space $x \in \T^d$ (the
$d$-dimensional unit torus), and velocity $v \in \R^d$. The operator
$\mathcal{L}$ acts only on the $v$ variable, and it must typically be
the generator of a stochastic semigroup for our method to work. We
give explicit results for $\mathcal{L}$ equal to the linear relaxation
Boltzmann operator (sometimes known as linear BGK operator), and for
$\mathcal{L}$ equal to the linear Boltzmann operator (see below for a
full description). We also consider the equations posed on the whole
space $(x, v) \in \R^d \times \R^d$ with a confining potential $\Phi$:
\begin{equation*}
  \partial_t f + v \cdot \nabla_x f
  -\left(\nabla_x \Phi \cdot \nabla_v f\right)
  = \mathcal{L} f.
\end{equation*}
We are able to give exponential convergence results on the
$d$-dimensional torus, or with confining potentials growing at least
quadratically at $\infty$, always in total variation or weighted total
variation norms (alternatively, $L^1$ or weighted $L^1$ norms). For
subquadratic potentials we give algebraic convergence rates, again in
the same kind of weighted $L^1$ norms. Some results were already
available for these equations \cite{CCG03,EL10,NM06,DMS15,H07, Me17}. We will give a more detailed account of them after we describe them
more precisely. Previous proofs of convergence to equilibrium used
strongly weighted $L^2$ norms (typically with a weight which is the
inverse of a Gaussian), so one advantage of our method is that it
directly yields convergence for a much wider range of initial
conditions. The result works, in particular, for initial conditions
with slow decaying tails, and for measure initial conditions with very
bad local regularity. The method gives also existence of stationary
solutions under quite general conditions; in some cases these are
explicit and easy to find, but in other cases they can be
nontrivial. We also note that our results for subquadratic potentials
are to our knowledge new. Apart from these new results, our aim is to
present a new application of a probabilistic method, using mostly PDE
arguments, and which is probably useful for a wide range of models.

The study of the speed of relaxation to equilibrium for kinetic
equations is a well known problem, both for linear and nonlinear
models. The central obstacle is that dissipation happens only on the
$v$ variable via the effect of the operator $\mathcal{L}$, while only
transport takes place in $x$. The transport then ``mixes'' the
dissipation into the $x$ variable, and one has to find a way to
estimate this effect. The theory of hypocoercivity was developed in
\cite{V09,H07,HN04} precisely to overcome these problems for linear
operators. In a landmark result, \cite{DV05} proved that the full
nonlinear Boltzmann equation converges to equilibrium at least at an
algebraic rate. Exponential convergence results for the (linear)
Fokker-Planck equation were given in \cite{DV01}, and a theory for a
range of linear kinetic equations has been given in \cite{DMS15}. All
of these results give convergence in exponentially weigthed $L^2$
norms or $H^1$ norms; convergence to equilibrium in weigthed $L^1$
norms can then be proved for several kinetic models by using the
techniques in \cite{GMM}. There are also works which deal directly
with the $L^1$ theory via abstract semigroup methods. We mention \cite{BS13, M14}  which give results for linear scattering equations on
the torus with spatially degenerate jump rates.

\medskip

Let us describe our equations more precisely. The \emph{linear
  relaxation Boltzmann equation} is given by
\begin{equation}
  \label{LRB}
  \partial_t f + v \cdot \nabla_x f -\left(\nabla_x \Phi \cdot
    \nabla_v f\right) = \mathcal{L}^+ f -f,
\end{equation}
where
\[ \mathcal{L}^+ f =\left( \int f(t,x,u) \mathrm{d}u\right) \mathcal{M}(v), \]
and $\M(v) := (2\pi)^{-d/2} \exp(-|v|^2 / 2)$. The function
$f = f(t,x,v)$ depends on time $t \geq 0$, space $x \in \R^d$, and
velocity $v \in \R^d$, and the potential $\Phi \: \R^d \to \R$ is a
$\mathcal{C}^2$ function of $x$. Alternatively, we consider this equation on
the torus; that is, for $x \in \T^d$, $v \in \R^d$, assuming periodic
boundary conditions. In that case we omit $\Phi$ (which corresponds to
$\Phi = 0$ in the above equation):
\begin{equation}
  \label{LRB-torus}
  \partial_t f + v \cdot \nabla_x f = \mathcal{L}^+ f -f.
\end{equation}
This simple equation is well studied in kinetic theory and can be
thought of as a toy model with similar properties to either the
non-linear BGK equation or linear Boltzmann equation. It is also one
of the simplest examples of a hypocoercive equation.  Convergence to
equilibrium in $H^1$ for this equation has been shown in \cite{CCG03},
at a rate faster than any function of $t$. It was then shown to
converge exponentially fast in both $H^1$ and $L^2$ using
hypocoercivity techniques in \cite{H07, NM06, DMS15}.

The linear Boltzmann equation is of a similar type:
\begin{equation}
  \label{LBEharris}
  \partial_t f+ v \cdot \nabla_x f-\left(\nabla_x \Phi \cdot \nabla_v
    f\right)  = Q(f, \mathcal{M}),
\end{equation}
where $\Phi$ is a $\mathcal{C}^2$ potential and
$\M(v) := (2\pi)^{-d/2} \exp(-|v|^2 / 2)$ as before, and $Q$ is the
Boltzmann operator
\[ Q(f,g) = \int_{\mathbb{R}^d}\int_{\mathbb{S}^{d-1}} B(|v- v_*|,
  \sigma)\left(f(v')g(v_*')-f(v)g(v_*)
  \right) \d\sigma \d v_*, \]
\[ v' = \frac{v+v_*}{2} + \frac{|v-v_*|}{2}\sigma, \quad v'_* =
  \frac{v+v_*}{2}- \frac{|v-v_*|}{2}\sigma,
\]
and $B$ is the \emph{collision kernel}. We always assume that $B$ is a
hard kernel and can be written as a product
\begin{equation}
  \label{eq:Bsplits}
  B(|v-v_*|, \sigma) =
  |v-v_*|^\gamma\ b\left( \sigma \cdot \frac{v-v_*}{|v-v_*|} \right),
\end{equation}
for some $\gamma \geq 0$ and $b$ integrable and uniformly positive on
$[-1,1]$; that is, there exists $C_b > 0$ such that
\begin{equation}
  \label{eq:b-hyp}
  b(z) \geq C_b
  \qquad \text{for all $z \in [-1,1]$.}
\end{equation}
This assumption includes for example the physical hard-spheres
collision kernel, for which $b \equiv 1$. The so-called
\emph{non-cutoff} kernels, for which $b$ is not integrable, are not
considered in this work.

As before, alternatively we consider the same equation posed for $x
\in \T^d$, $v \in \R^d$, without any potential $\Phi$:
\begin{equation}
  \label{LBEharris-torus}
  \partial_t f+ v \cdot \nabla_x f = Q(f, \mathcal{M}).
\end{equation}
This equation models gas particles interacting with a background
medium which is already in equilibrium. Moreover, it has been used in describing many other systems like radiative transfer, neutron transportation, cometary flow and dust particles. The spatially homogeneous case
has been studied in \cite{LMT, BCL15, CEL17, LM17}. The kinetic equations
\eqref{LBEharris} or \eqref{LBEharris-torus} fit into the general
framework in \cite{NM06, DMS15}, so convergence to equilibrium in
weighted $L^2$ norms may be proved by using the techniques described
there.

We denote by $\P(\Omega)$ the set of probability measures on a set
$\Omega \subseteq \R^k$ (that is, the probability measures defined on
the Borel $\sigma$-algebra of $\Omega$). We state our main results on
the torus, and then on $\R^d$ with a confining potential:

\begin{thm}[Exponential convergence results on the torus]
  \label{thm:main-torus}
  Suppose that $t \mapsto f_t$ is the solution to \eqref{LRB-torus} or
  \eqref{LBEharris-torus} with initial data
  $f_0 \in \P(\T^d \times \R^d)$. In the case of equation
  \eqref{LBEharris-torus} we also assume \eqref{eq:Bsplits} with
  $\gamma \geq 0$ and \eqref{eq:b-hyp}. Then there exist constants
  $C>0, \lambda >0$ (independent of $f_0$) such that
  \[
    \| f_t - \mu \|_* \leq Ce^{-\lambda t} \| f_0 - \mu \|_*,
  \]
  where $\mu$ is the only equilibrium state of the corresponding
  equation in $\P(\T^d \times \R^d)$ (that is, $\mu(x,v) =
  \M(v)$). The norm $\|\cdot\|_*$ is just the total variation norm
  $\|\cdot\|_{\mathrm{TV}}$ for equation \eqref{LRB-torus},
  \begin{equation*}
    \| f_0 - \mu \|_* = \|f_0 - \mu\|_{\mathrm{TV}}
    := \int_{\R^d} \int_{\T^d} |f_0 - \mu| \d x \d v
    \qquad \text{for equation \eqref{LRB-torus}},
  \end{equation*}
  and it is a weigthed total variation norm in the case of equation
  \eqref{LBEharris-torus}:
  \begin{equation*}
    \| f_0 - \mu \|_* = \int_{\R^d} \int_{\T^d}
    (1+|v|^2) |f_0-\mu| \d x \d v
    \qquad \text{for equation \eqref{LBEharris-torus}}.
  \end{equation*}
\end{thm}

\begin{thm}[Exponential convergence results with a confining potential]
  \label{thm:main-confining}
  Suppose that $t \mapsto f_t$ is the solution to \eqref{LRB} or
  \eqref{LBEharris} with initial data $f_0 \in \P(\R^d \times
  \R^d)$ and a potential $\Phi \in \mathcal{C}^2(\R^d)$ which is
  bounded below, and satisfies
  \[
    x\cdot \nabla_x \Phi(x) \geq \gamma_1 |x|^2 + \gamma_2 \Phi(x) - A,
    \qquad x \in \R^d,
  \]
  for some positive constants $\gamma_1, \gamma_2, A$.  Define $\langle x \rangle = \sqrt{1+|x|^2}$. In the case of
  equation \eqref{LBEharris-torus} we also assume \eqref{eq:Bsplits}
  with $\gamma \geq 0$, 
  \eqref{eq:b-hyp} and
  \[ x \cdot \nabla_x \Phi(x) \geq \gamma_1 \langle x \rangle^{\gamma+2} + \gamma_2 \Phi(x) -A, \]  
    for some positive constants $\gamma_1, \gamma_2, A$. Then there exist constants $C>0, \lambda >0$
  (independent of $f_0$) such that
  \[
    \| f_t - \mu\|_* \leq Ce^{-\lambda t} \| f_0 - \mu \|_*,
  \]
  where $\mu$ is the only equilibrium state of the corresponding
  equation in $\P(\R^d \times \R^d)$,
  \[ \mathrm{d}\mu = \mathcal{M}(v)e^{-\Phi(x)}\mathrm{d}v\mathrm{d}x. \]
  The norm $\|\cdot\|_*$ is a weighted total variation norm defined by
  \[
    \| f_t - \mu \|_* := \int \left( 1+ \frac{1}{2}|v|^2
      +\Phi(x) +|x|^2 \right)|f_t - \mu| \d v \d x.
  \]
\end{thm}

In all results above the constants $C$ and $\lambda$ can be explicitly
estimated in terms of the parameters appearing in the equation by
following the calculations in the proofs. We do not give them
explicitly since we do not expect them to be optimal, but they are
nevertheless completely constructive.

We also look at Harris type theorems with weaker controls on moments
to give analogues of all our theorems when the confining potential is
weaker and give algebraic rates of convergence with rates depending on
the assumption we make on the confining potential. Subgeometric
convergence for kinetic Fokker-Planck equations with weak confinement
has been shown in \cite{DFG09, BCG08, C18}. To our knowledge this is
the only work showing this type of convergence in a quantitative way
for the equations we present.

\begin{thm}[Subgeometric convergence results with weak confining potentials] \label{thm:main-subgeometric}
  Suppose that $t \mapsto f_t$ is the solution to \eqref{LRB} in the whole space with a confining potential $\Phi \in \mathcal{C}^2(\R^d)$ which is bounded below. Define $\langle x \rangle = \sqrt{1+|x|^2}$. Assume that for some
  $\beta$ in $(0, 1)$ the confining potential satisfies
\[ x\cdot\nabla_x\Phi(x) \geq \gamma_1 \langle x \rangle^{2\beta} + \gamma_2\Phi(x) -A,    \] for some positive constants $\gamma_1, \gamma_2, A$. Then there exists a constant $C>0$ such that
\[ \|f_t-\mu\|_{\mathrm{TV}} \leq \min \left\{ \|f_0-\mu\|_{\mathrm{TV}},
    \
    C \int f_0(x,v)\left(1+\frac{1}{2}|v|^2 + \Phi(x) + |x|^2\right)
    (1+t)^{-\beta/(1-\beta)}\right\}.\]
Similarly if $t \mapsto f_t$ is the solution to \eqref{LBEharris} in the whole
  space, satisfies \eqref{eq:Bsplits}
  \eqref{eq:b-hyp} and 
\[ x \cdot \nabla_x \Phi(x) \geq \gamma_1 \langle x \rangle^{1+\beta} + \gamma_2 \Phi(x) -A, \quad \Phi(x) \le \gamma_3 \langle x \rangle^{1+\beta},\]
for some positive constants $\gamma_1, \gamma_2, \gamma_3 A, \beta$, then there exists a constant $C>0$ such that 
\[ \|f_t-\mu\|_{\mathrm{TV}} \leq \min \left\{ \|f_0-\mu\|_{\mathrm{TV}},
    \
    C \int f_0(x,v)\left(1+\frac{1}{2}|v|^2 + \Phi(x) + |x| \right)
    (1+t)^{-\beta}\right\}.\]

\end{thm}

\medskip

We carry out all of our proofs using variations of Harris's Theorem
from probability. Harris's Theorem originated in the paper \cite{H56}
where Harris gave conditions for existence and uniqueness of a steady
state for Markov processes. It was then pushed forward by Meyn and
Tweedie \cite{MT93} to show exponential convergence to
equilibrium. The last paper \cite{HM11} gives an efficient way of
getting quantitative rates for convergence to equilibrium once you
have quantitatively verified the assumptions, we use this version of
the result. Harris's Theorem says, broadly speaking, that if you have
a good confining property and some uniform mixing property in the
centre of the state space then you have exponentially fast convergence
to equilibrium in a weighted total variation norm. We give the precise
statement in the next section.

Harris's Theorem is based on the classical Doeblin's Theorem. The use
of such tools to analyse PDEs was first introduced in \cite{BL55,
  BL57} to study integro-differential equations for scatterers. These
results are very broad and include both models with non-equilibrium
steady states and spatially degenerate jump rates. Doeblin's Theorem
has also been used in \cite{BLP79} to show ergodicity for a kinetic
equation. In all these papers the authors do not seek to find explicit
rates. Harris's Theorem has already been used to show convergence to
equilibrium for some kinetic equations. In \cite{HMS02}, the authors
show convergence to equilibrium for the kinetic Fokker-Planck equation
with non-quantitative rates. In \cite{CELMM16}, the authors show
quantitative exponential convergence to a non-equilibrium steady state
for some non-linear kinetic equations on the torus using Doeblin's
Theorem.

This method is also applicable to some integro-PDEs describing several biological and
physical phenomena. In \cite{Gabriel18}, Doeblin's argument is used to
show exponential relaxation to equilibrium for the conservative
renewal equation which is a common model in population dynamics, often
referred as the McKendrick-von Foerster equation. In \cite{CY19}, the authors
show existence of a spectral gap property in the linear
(no-connectivity) setting for elapsed-time structured neuron networks
by using Doeblin's Theorem. Relaxation to the stationary state for the
original nonlinear equation is then proved by a perturbation argument where the 
non-linearity is weak. Moreover, in \cite{DG2017} the authors
consider a nonlinear model which is derived from mean-field
description of an excitatory network made up of leaky
integrate-and-fire neurons. In the case of weak connectivity, the
authors demonstrate the uniqueness of a stationary state and its
global exponential stability by using Doeblin's type of contraction argument 
for the linear case. Also in \cite{BCG2017}, the
authors extend similar ideas to obtain quantitative estimates in total
variation distance for positive semigroups, that can be
non-conservative and non-homogeneous. They provide a speed of
convergence for periodic semigroups and new bounds in the homogeneous
setting.

Using Harris's Theorem gives an alternative and very different
strategy for proving quantitative exponential decay to equilibrium. It
allows us to look at hypocoercive effects on the level of stochastic
processes and to look at specific trajectories which might allow one
to produce quantitative theorems based on more trajectorial
intuition. Another difference is that the confining behaviour is shown
here by exploiting good behaviour of moments rather than a
Poincar\'{e} inequality, this means looking at point wise bounds
rather than integral controls on the operator. These are often
equivalent for time reversible processes \cite{BCG08, CG14} and have
advantages and disadvantages. However, the condition on the moments
used here might be much easier to verify in the case where the
equilibrium state cannot be made explicit. This is the motivation
behind \cite{BL55, CELMM16}. These works also allow us to look at a
large class of initial data. We only need $f_0$ to be a probability
measure where $\|f_0 - \mu\|$ is finite. Harris's Theorem has a
restriction which is that we can only consider Markov processes. Many
kinetic equations are linear Markov processes but this excludes the
study of linearized non-linear equations which are not necessarily
mass preserving.

The plan of the paper is as follows. We introduce Harris's Theorem in
Section \ref{sec:Harris}. Then we have a section for each of our
equations where we prove our results.

\section{Harris's Theorem}
\label{sec:Harris}

Now let us be more specific about Harris's Theorem.  We give the
theorems and assumption as in the setting of \cite{HM11} where they
make it clear how the rates depend on those in the assumptions.
Markov operators can be defined by means of \emph{transition
  probability functions}. We always assume that
$(\Omega, \mathcal{S})$ is a measurable space. A function
$S \: \Omega \times \mathcal{S} \to \R$ is a transition probability
function on a finite measure space if $S(x, \cdot)$ is a probability
measure for every $x$ and $x \mapsto S(x, A)$ is a measurable function
for every $A \in \mathcal{S}$.  We can then define $\mathcal{P}$, the
associated stochastic operator on probability measures by
\[ \mathcal{P} \mu(\cdot) = \int_\Omega \mu(\mathrm{d}x) S(x, \cdot).\] 
In a similar way we can define the action of $S$ on functions (observables) by
\[ (\mathcal{P}^*\psi)(x) :=\int_\Omega \psi(y) S(x, \mathrm{d}y). \]
Since
we are looking at a process we have Markov transition kernel $S_t$ for
each $t>0$. We also define $\mathcal{P}_t$ from $S_t$ as above. In our
situation $\mathcal{P}_t\mu$ is the weak solution to the PDE with
initial data $\mu$. If we define $\mathcal{M}(\Omega)$ as the space of
finite measures on $(\Omega, \mathcal{S})$ then we have that
$\mathcal{P}_t$ is a \emph{linear} map
\[\mathcal{P}_t : \mathcal{M}(\Omega) \rightarrow
  \mathcal{M}(\Omega). \] From the conditions on $S_t$ we see that
$\mathcal{P}_t$ will be \emph{linear, mass preserving} and
\emph{positivity preserving}. 

We can define the \emph{forwards operator} $\UU$, associated
to $S_t$ as the operator which satisfies
\begin{equation}
  \label{eq:forwards_op}
  \left.\frac{\mathrm{d}}{\mathrm{d}t} \mathcal{P}^*_t
    \psi\right|_{t=0} = \UU\psi,
\end{equation}
for all $\psi \in C_c^\infty(\Omega)$.

We begin by looking at Doeblin's Theorem. Harris's Theorem is a natural successor to Doeblin's Theorem. Harris's and Doeblin's theorems are normally stated for a fixed time $t_*$. In our theorems we work to choose an appropriate $t_*$.
\begin{hyp}[Doeblin's Condition]
  \label{hyp:doeblin}
  We assume $(\mathcal{P}_{t})_{t \geq 0}$ is a stochastic semigroup, coming from
  a Markov transition kernel, and that there exists $t_* > 0$, a
  probability distribution $\nu$ and $\alpha \in (0,1)$ such that for
  any $z$ in the state space we have
  \[ \mathcal{P}_{t_*} \delta_z \geq \alpha \nu. \]
\end{hyp}

Using this we prove

\begin{thm}[Doeblin's Theorem]
  \label{thm:Doeblin}
  If we have a stochastic semigroup $(\mathcal{P}_{t})_{t \geq 0}$
  satisfying Doeblin's condition (Hypothesis \ref{hyp:doeblin}) then
  for any two measures $\mu_1$ and $\mu_2$ and any integer $n \geq 0$
  we have that
\begin{equation} \label{Doeblin1}
  \|\mathcal{P}_{t_*}^n\mu_1 - \mathcal{P}_{t_*}^n \mu_2 \|_{\mathrm{TV}} \leq (1-\alpha)^n \|\mu_1-\mu_2\|_{\mathrm{TV}}.
\end{equation} 
As a consequence, the semigroup has a unique equilibrium probability
measure $\mu_*$, and for all $\mu$
\begin{equation}
  \label{Doeblin2}
  \|\mathcal{P}_{t} (\mu -\mu_*)\|_{\mathrm{TV}}
  \leq \frac{1}{1-\alpha} e^{-\lambda t} \|\mu -
  \mu_*\|_{\mathrm{TV}},
  \qquad t \geq 0,
\end{equation}
where
\[ \lambda := \frac{\log(1-\alpha)}{t_*} >0 .\]
\end{thm}

\begin{proof}
  This proof is classical and can be found in various versions in
  \cite{HM11} and many other places.

Firstly we show that if $\mathcal{P}_t\delta_z \geq \alpha \nu$ for every $z$, then we also have $\mathcal{P}_t \mu \geq \alpha \nu$ for every probability measure $\mu$. Here since $\mathcal{P}_t$ comes from a Markov transition kernel we have 
\[ \mathcal{P}_t \delta_z (\cdot)= \int S_t(z', \cdot)\delta_z(\mathrm{d}z') = S_t(z,\cdot).   \] Therefore our condition says that
\[ S_t(z, \cdot) \geq \alpha \nu(\cdot) \] for every $z$. Therefore,
\[ \mathcal{P}_t \mu(\cdot) = \int S_t(z, \cdot)\mu(\mathrm{d}z) \geq \alpha \int \nu(\cdot)\mu(\mathrm{d}z) = \alpha \nu(\cdot). \]

By the triangle inequality we have
\[\| \mathcal{P}_{t_*} \mu_1 - \mathcal{P}_{t_*} \mu_2 \|_{\mathrm{TV}}
\leq
\| \mathcal{P}_{t_*} \mu_1 - \alpha \nu \|_{\mathrm{TV}}
+ \| \mathcal{P}_{t_*} \mu_2  - \alpha \nu \|_{\mathrm{TV}}. \]
Now, since $\mathcal{P}_{t_*} \mu_1 \geq \alpha \nu$, we can write
\[\| \mathcal{P}_{t_*} \mu_1 - \alpha \nu \|_{\mathrm{TV}}
=
\int (\mathcal{P}_{t_*} \mu_1 - \alpha \nu)
=
\int \mu_1 - \alpha
= 1-\alpha, \]
due to mass conservation, and similarly for the term
$\| \mathcal{P}_{t_*} \mu_2 - \alpha \nu \|_{\mathrm{TV}}$. This gives 
\[\| \mathcal{P}_{t_*} \mu_1 - \mathcal{P}_{t_*} \mu_2 \|_{\mathrm{TV}}
\leq
2(1-\alpha)
= (1-\alpha) \|\mu_1 - \mu_2\|_{\mathrm{TV}} \]
if $\mu_1, \mu_2 $ have disjoint support.  By
homogeneity, this inequality is obviously also true for any nonnegative $\mu_1, \mu_2$ having disjoint
support with $\int \mu_1 = \int \mu_2$. We obtain the inequality in
general for any $\mu_1, \mu_2$ with the same
integral by writing
$\mu_1 - \mu_2 = (\mu_1 - \mu_2)_+ - (\mu_2 - \mu_1)_+$, which is a difference of nonnegative measures with the same integral. This proves \begin{equation} \label{Doeblin-contractivity}
\|\mathcal{P}_{t_*}\mu_1 - \mathcal{P}_{t_*}\mu_2\|_{\mathrm{TV}} \leq (1-\alpha) \|\mu_1-\mu_2\|_{\mathrm{TV}}.
\end{equation} We then iterate this to obtain \eqref{Doeblin1}. 
The contractivity \eqref{Doeblin-contractivity} shows that the operator
$\mathcal{P}_{t_*}$ has a unique fixed point, which we
call $\mu_*$. In fact, $\mu_*$ is a stationary state of the whole semigroup since for all $s \geq 0$ we have
\[\mathcal{P}_{t_*} \mathcal{P}_{s} \mu_* = \mathcal{P}_{s} \mathcal{P}_{t_*} \mu_* = \mathcal{P}_{s} \mu_*, \]
which shows that $ \mathcal{P}_{s} \mu_*$ (which is again a probability measure) is also a stationary state of $\mathcal{P}_{t_*}$; due to uniqueness,
\[ \mathcal{P}_{s} \mu_* = \mu_*.\]
Hence the only stationary state of $\mathcal{P}_{t}$ must be
$\mu_*$, since any stationary state of $\mathcal{P}_{t}$ is in particular a stationary state of $\mathcal{P}_{t_*}$. 

In order to show \eqref{Doeblin2}, for any probability measure $\mu$ and any $t \geq 0$ we write
\[k := \left \lfloor{\frac{t}{t_*}} \right \rfloor, \]
(where $\lfloor \cdot \rfloor$ denotes the integer part) so that
\[\frac{t}{t_*} - 1 < k \leq \frac{t}{t_*}. \]
Then,
\begin{multline*}
\| \mathcal{P}_{t} (\mu - \mu_*) \|_{\mathrm{TV}}
= \| \mathcal{P}_{t - k t_*} \mathcal{P}_{k
	t_*} (\mu-\mu_*) \|_{\mathrm{TV}}
\leq
\| \mathcal{P}_{k t_*} (\mu-\mu_*) \|_{\mathrm{TV}}
\\
\leq
(1-\alpha)^{k} \| \mu-\mu_* \|_{\mathrm{TV}}
\leq
\frac{1}{1-\alpha}
\exp \left( t \log \left ( \frac{1-\alpha}{t_*} \right ) \right) \| \mu-\mu_* \|_{\mathrm{TV}}.
\qedhere
\end{multline*}
\end{proof}
Harris's Theorem extends this to the setting where we cannot prove
minorisation uniformly on the whole of the state space. The idea is to
use the argument given above on the centre of the state space then
exploit the Lyapunov structure to show that any stochastic process
will return to the centre infinitely often.

We make two assumptions on the behaviour of $\mathcal{P}_{t_*}$, for some fixed $t_*$:
\begin{hyp} [Lyapunov condition]\label{confinement}
There exists some function $V \hspace{2pt}:\hspace{2pt} \Omega \rightarrow [0, \infty)$ and constants $D \geq 0, \alpha \in (0,1)$ such that
\[ (\mathcal{P}^*_{t_*}V)(z) \leq \alpha V(z) + D. \]
\end{hyp}

\begin{remark}
  We use the name Lyapunov condition as it is the standard name used
  for this condition in probability literature. However, we should
  stress this condition is not related to the Lyapunov method for
  proving convergence to equilibrium. We do not prove monotonicity of
  a functional.
\end{remark}

\begin{remark}
  \label{rem:Lyapunov-equivalent}
  In our situation, we have an equation on the law
  $f(t) \equiv \PP_t f_0$. This is equivalent to the statement
  \begin{equation}
    \label{eq:step-Lyap}
    \int_{\Omega} f(t,z) V(z) \mathrm{d}z \leq \alpha \int_{\Omega} f_0(z) V(z)
    \mathrm{d}z + D.
  \end{equation}
  We normally verify this by showing that
  \[ \frac{\mathrm{d}}{\mathrm{d}t} \int_{\Omega} f(t,z)V(z) \mathrm{d}z \leq -
    \lambda \int_{\Omega} f(t,z)V(z) \mathrm{d}z +  K,\] for some positive
  constants $K$ and $\lambda$, which then implies \eqref{eq:step-Lyap}
  for $\alpha = e^{-\lambda t}$ and $D = \frac{K}{\lambda}(1 -
  e^{-\lambda t}) \leq K t$. Equivalently, one can show that
  \begin{equation*}
    \UU V \leq - \lambda V + K,
  \end{equation*}
  where $\UU$ is the forwards operator defined in
  \eqref{eq:forwards_op}.
\end{remark}

The idea behind verifying the Lyapunov structure in our case comes
from \cite{HMS02} where they use similar Lyapunov structures for the
kinetic Fokker-Planck equation. When we work on the torus the Lyapunov
structure is only needed in the $v$ variable and the result is purely
about how moments in $v$ are affected by the collision operator.

The next assumption is a minorisation condition as in Doeblin's
Theorem
\begin{hyp}
  \label{minorisation1}
  There exists a probability measure $\nu$ and a constant
  $\beta \in (0,1)$ such that
  \[ \inf_{z \in \mathcal{C}} \mathcal{P}_{t_*}\delta_z \geq \beta \nu, \] where
  \[ \mathcal{C} =  \{ z \: V(z) \leq R \} \]
  for some $R > 2D/(1-\alpha)$.
\end{hyp}

\begin{remark}
  Production of quantitative lower bounds as a way to quantify the
  \emph{positivity} of a solution has been proved and used in kinetic
  theory before. For example it is an assumption required for the
  works of Desvillettes and Villani \cite{DV01, DV05}. Such lower
  bounds have been proved for the non-linear Boltzmann equation in
  \cite{M05, B15, B15b}.
\end{remark}

This second assumption is more challenging to verify in our
situations. Here we use a strategy based on our observation about how
noise is transferred from the $v$ to the $x$ variable as described
earlier. The actual calculations are based on the PDE governing the
evolution and iteratively using Duhamel's formula.

We define a weighted total variation norm on measures for each $a$ by:
\[\|\mu_1- \mu_2\|_{V,a} = \int ( 1+ a V(z)) |\mu_1 - \mu_2|( \mathrm{d}z). \]
\begin{thm}[Harris's Theorem as in \cite{HM11}] \label{harris}
If Hypotheses \ref{confinement} and \ref{minorisation1} hold then there exist $\bar{\alpha} \in (0,1)$ and $a >0$ such that
\begin{equation}
\| \mathcal{P}_{t_*}\mu_1- \mathcal{P}_{t_*}\mu_2\|_{V,a} \leq \bar{\alpha} \|\mu_1- \mu_2\|_{V,a}.
\end{equation}
 Explicitly if we choose $\beta_0 \in (0, \beta)$ and $\alpha_0 \in ( \alpha + 2D/R, 1)$ then we can set $ \gamma = \beta_0/K$ and $ \bar{\alpha} = \max \left \{ 1-(\beta - \beta_0), (2 + R \gamma \alpha_0)/(2+R \gamma)\right\}.$
\end{thm}
\begin{remark}
We have that
\[ \min \{ 1, a\} \|\mu_1- \mu_2\|_{V,1} \leq \|\mu_1- \mu_2\|_{V,a} \leq \max \{ 1, a \} \|\mu_1- \mu_2\|_{V,1}.\]
 We can also iterate Theorem \ref{harris} to get
\[ \| \mathcal{P}_{nt_*} \mu_1- \mathcal{P}_{nt_*} \mu_2\|_{V,a} \leq \bar{\alpha}^n \| \mu_1- \mu_2\|_{V,a}. \] Therefore we have that
\[  \| \mathcal{P}_{nt_*} \mu_1- \mathcal{P}_{nt_*} \mu_2\|_{V,1} \leq \bar{\alpha}^n \frac{\max \{ 1, a \} }{ \min \{ 1, a \}} \| \mu_1- \mu_2\|_{V,1}.\]
\end{remark}

\begin{remark}
In this paper we always consider functions $V$ where $V(z) \rightarrow \infty$ as $|z|\rightarrow \infty$. In this case, we can replace $\mathcal{C}$ in Hypothesis \ref{minorisation1} with some ball of radius $R'$ which will contain $\mathcal{C}$.
\end{remark}

Doeblin's Theorem corresponds to the irreducibility property for
Markov processes in the bounded state space. But when the state space
is unbounded it is expected that the process may drift arbitrarily far
away and we cannot prove a uniform minorisation condition on the whole
state space. Harris's Theorem is one way to extend the ideas of
Doeblin to the unbounded state space by finding a Lyapunov functional
with small level sets to show that transition probabilities of the
process converge towards a unique invariant measure. Therefore
Harris's Theorem is based on providing a combination of a minorisation
and a geometric drift conditions. The minorisation condition can be
thought as finding a bound on the probability of transitioning in one
step from any initial state to some specified region which corresponds
to the small level sets of the Lyapunov functional.

Another way to prove convergence towards a unique invariant measure
for a Markov process is to show the Markov semigroup has the
\emph{strong Feller property}, meaning that the semigroup maps bounded
measurable functions to bounded continuous functions. We refer the
reader to \cite{H10} for further comments on this.

There are versions of Harris's Theorem adapted to weaker Lyapunov
conditions which give subgeometric convergence \cite{DFG09}. We use
the following theorem which can be found in Section 4 of \cite{H10}.
\begin{thm}[Subgeometric Harris's Theorem]
  Given the forwards operator, $\UU$, of our Markov semigroup
  $\mathcal{P}$, suppose that there exists a continuous function $V$
  valued in $[1,\infty)$ with pre compact level sets such that
  \[ \UU V \leq K - \phi(V), \] for some constant $K$ and some
  strictly concave function
  $\phi: \mathbb{R}_+ \rightarrow \mathbb{R}$ with $\phi(0) = 0$ and
  increasing to infinity. Assume that for every $C >0$ we have the
  minorisation condition like Hypothesis \ref{minorisation1}. i.e. for
  some $t_*$ a time and $\nu$ a probability distribution and
  $\alpha \in (0,1)$, then for all $z$ with $V(z) \leq C$:
  \[ \mathcal{P}_{t_*} \delta_z \geq \alpha \nu. \] With these
  conditions we have that
  \begin{itemize}
  \item There exists a unique invariant measure $\mu_*$ for the Markov
    process and it satisfies
    \[ \int \phi(V(z)) \mathrm{d}\mu_* \leq K. \]
    
  \item Let $H_\phi$ be the function defined by 
    \[ H_\phi = \int_1^u \frac{\mathrm{d}s}{\phi(s)}.\] Then there
    exists a constant $C$ such that
    \[ \|\mathcal{P}_t \mu - \mu_* \|_{\mathrm{TV}}
      \leq
      \frac{C\mu(V)}{H_\phi^{-1}(t)} + \frac{C}{(\phi \circ
        H_\phi^{-1})(t)}
    \]
    holds for every $\mu$ where $\mu(V) = \int V(z) \mu(\mathrm{d}z)$.
  \end{itemize}
\end{thm}

We will apply the subgeometric Harris's Theorem to the PDEs we study
to show convergence when only a weaker confinement condition is
available.

\section{The linear relaxation Boltzmann equation}

\subsection{On the torus}

This is the simplest operator on the torus, so we do not in fact need
to use Harris's Theorem. We can instead use Doeblin's Theorem where we
have a uniform minorisation condition.

We consider
\begin{equation}
  \label{eq:LBGK}
  \p_t f + v \cdot \nabla_xf = \mathcal{L}f,
\end{equation}
posed for $(x,v) \in \T^d \times \R^d$, where $\T^d$ is the
$d$-dimensional torus of side $1$ and
\begin{equation}
  \label{eq:L-LBGK}
  \mathcal{L}f (x,v) := \mathcal{L}^+ f(x,v) - f(x,v)
  := \left( \ird f(x,u) \d u \right) \mathcal{M}(v) - f(x,v),
\end{equation}
which is a well defined operator from $L^1(\T^d \times \R^d)$ to
$L^1(\T^d \times \R^d)$, and can also be defined as an operator from
$\M(\T^d \times \R^d)$ to $\M(\T^d \times \R^d)$ with the same
expression (where $\ird f(x,u) \d u$ now denotes the marginal of the
measure $f$ with respect to $u$). We define $(T_t)_{t \geq 0}$ as the
transport semigroup associated to the operator $-v \cdot \nabla_x f$
in the space of measures with the bounded Lipschitz topology (see for
example \cite{CCC13}); that is, $t \mapsto T_t f_0$ solves the
equation $\p_t f + v \cdot \nabla_x f = 0$ with initial condition
$f_0$. In this case one can write $T_t$ explicitly as
\begin{equation}
  \label{eq:T-LBGK}
  T_t f_0(x,v) = f_0(x - tv, v).
\end{equation}
Using Duhamel's formula repeatedly one can obtain that, if $f$ is a
solution of \eqref{eq:LBGK} with initial data $f_0$, then
\begin{equation}
  \label{eq:Duhamel}
  e^t f_t \geq \int_0^t \int_0^s T_{t-s} \mathcal{L}^+ T_{s-r} \mathcal{L}^+ T_r f_0
  \d r \d s.
\end{equation}

We will now check two properties, which we list as lemmas. The first
one says that the operator $\mathcal{L}$ always allows jumps to any small
velocity. We always use the notation $\1_A$ to denote the characteristic
function of a set $A$ (if $A$ is a set), or the function which is $1$
where the condition $A$ is met, and $0$ otherwise (if $A$ is a
condition).
\begin{lem}
  \label{lem:HL}
  For all $\delta_L > 0$ there exists $\alpha_L > 0$ such that for all
 $g \in \mathcal{P}(\mathbb{T}^d \times \mathbb{R}^d)$ we have
  \begin{equation}
    \label{eq:HL}
    \mathcal{L}^+ g(x,v) \geq
    \alpha_L \left( \ird g(x,u) \d u \right) \1_{\{|v| \leq \delta_L\}}
  \end{equation}
  for almost all $(x, v) \in \T^d \times \R^d$.
\end{lem}

\begin{proof}
  Given any $\delta_L$ it is enough to choose $\alpha_L := \M(v)$ for
  any $v$ with $|v| = \delta_L$.
\end{proof}

The second one is regarding to the behaviour of the transport part alone. It
says that if we start at any point inside a ball of radius $R$, and we
are allowed to start with any small velocity, then we can reach any
point in the ball of radius $R$ with a predetermined bound on the
final velocity. We use $B(\delta)$
  to denote the open ball $\{x \in \R^d \mid |x| < \delta\}$, and in
  general we will use the notation $B(z,\delta)$ to denote the open ball of
  radius $\delta$ centered at $z \in \R^d$.The lemma is:
\begin{lem}
  \label{lem:HT}
  Given any time $t_0 > 0$ and radius $R > 0$ there exist $\delta_L, R'
  > 0$ such that for all $t \geq t_0$ it holds that
  \begin{equation}
    \label{eq:HT}
    \int_{B(R')} T_{t} \Big(\delta_{x_0}(x) \1_{\{|v| \leq \delta_L\}}\Big) \d v
    \geq
    \frac{1}{t^d} \1_{\{|x| \leq R \}}
    \qquad \text{for all $x_0$ with $|x_0| < R$.}
  \end{equation}
  In particular, if we take $R > \sqrt{d}$, there exist $\delta_L, R'
  > 0$ such that
  \begin{equation}
    \label{eq:HT*}
    \int_{B(R')} T_{t} \Big(\delta_{x_0}(x) \1_{\{|v| \leq \delta_L\}}\Big) \d v
    \geq
    \frac{1}{t^d}
    \qquad \text{for all $x_0 \in \T^d$.}
  \end{equation}
\end{lem}

\begin{proof}
  Take $t, R > 0$. We have
  \[ T_t \left(\delta_{x_0}(x) \1_{B(\delta_L)}(v) \right) =
    \delta_{x_0}(x-vt) \1_{B(\delta_L)}(v). \]   Integrating this and
  changing variables gives that
  \[ \int_{B(R')} T_t \left(
      \delta_{x_0}(x) \1_{B(\delta_L)}(v)
    \right)  \d v =
    \frac{1}{t^d}
    \int_{B(x, t R')} \delta_{x_0}(y) \1_{B(\delta_L)}\left( \frac{x-y}{t}
    \right) \mathrm{d}y.\]
  Since $|x-y| \leq |x|+|y|$ we have that
  \[ \1_{B(\delta_L)}\left( \frac{x-y}{t} \right)
    \geq \1_{B(\delta_L/2)}\left(\frac{x}{t}\right)\, \1_{B(\delta_L/2)}\left(\frac{y}{t}\right).\]
  Therefore if we take $\delta_L > 2R/t$ we have
  \[ \1_{B(\delta_L)}\left( \frac{x-y}{t} \right) \geq \1_{B(R)}(x)
    \1_{B(R)}(y).\]
  On the other hand, if we take $|x| < R$ and $R' > 2 R / t$ then
  \begin{equation*}
    B(x, t R') \supseteq B(x, 2R) \supseteq B(R).
  \end{equation*}
  Hence if $\delta_L > 2R / t$ and $R' > 2R / t$,
  \[ \int_{B(R')} T_t \left(\delta_{x_0}(x) \1_{B(\delta_L)}(v) \right) \d v \geq
    \frac{1}{t^d} \1_{B(R)}(x), \]
  which proves the result.
\end{proof}

\begin{lem}[Doeblin condition for the linear relaxation Botzmann equation on the torus]
  \label{lem:Doeblin_LBGK}
  For any $t_* > 0$ there exist constants $\alpha, \delta_L > 0$
  (depending on $t_*$) such that any solution $f$ to equation
  \eqref{eq:LBGK} with initial condition
  $f_0 \in \P(\T^d \times \R^d)$ satisfies
  \begin{equation}
    \label{eq:Doeblin_LBGK}
    f(t_*, x, v) \geq \alpha \1_{\{|v| \leq \delta_L\}},
  \end{equation}
  where the inequality is understood in the sense of measures.
\end{lem}

\begin{proof}
  It is enough to prove it for $f_0 := \delta_{(x_0,v_0)}$, where
  $(x_0,v_0) \in \T^d \times \R^d$ is an arbitrary point. From Lemma
  \ref{lem:HT} (with $R > \sqrt{d}$ and $t_0 := t_*/3$) we will use
  that there exists $\delta_L > 0$ such that
  \begin{equation*}
    \ird T_{t} \Big(\delta_{x_0}(x) \1_{\{|v| \leq \delta_L\}}\Big) \d v
    \geq
    \frac{1}{t^d}
    \qquad \text{for all $x_0 \in \T^d$, $t > t_0$}.
  \end{equation*}
  Also, Lemma \ref{lem:HL} gives an $\alpha_L > 0$ such that
  \begin{equation*}
    \mathcal{L}^+ g \geq \alpha_L
    \left( \ird g(x,u) \d u\right) \, \1_{\{|v| \leq \delta_L\}}.
  \end{equation*}
  Take any $r > 0$. Since $T_r f_0 = \delta_{(x_0 - v_0 r, v_0)}$,
  using this shows
  \begin{equation*}
    \mathcal{L}^+ T_r f_0 \geq
    \alpha_L \,
    \delta_{x_0 - v_0 r} (x)\, \1_{\{|v| \leq \delta_L\}}.
  \end{equation*}
  Hence, whenever
  $s-r > t_0$ we have
  \begin{align*}
    \mathcal{L}^+ T_{s-r} \mathcal{L}^+ T_r f_0
    &\geq
      \alpha_L
      \left(
      \ird T_{s-r} \mathcal{L}^+ T_r f_0 \d u
      \right)
      \1_{\{|v| \leq \delta_L\}}
    \\
    &\geq
      \alpha_L^2 \,
      \left(
      \ird T_{s-r} \big(
      \delta_{x_0 - v_0 r} (x) \, \1_{\{|u| \leq \delta_L\}}
      \big) \d u
      \right)
      \1_{\{|v| \leq \delta_L\}}
    \\
    &\geq
      \frac{1}{(s-r)^d} \alpha_L^2 \,
      \, \1_{\{|v| \leq \delta_L\}}.
  \end{align*}
  Finally, for the movement along the flow
  $T_{t-s}$, notice that
  \[ T_{t} \Big( \1_{\T^d} (x) \1_{\{|v| < \delta_L\}}(v) \Big) =
    \1_{\T^d}(x) \1_{\{|v| < \delta_L\}}(v)
  \qquad \text{for all $t \geq 0$}.\]
  This means that for
  all $t > s > r > 0$ such that $s-r > t_0$ we have
  \[  T_{t-s}\mathcal{L}^+ T_{s-r} \mathcal{L}^+ T_r f_0 \geq
    \frac{1}{(s-r)^d} \alpha_L^2
    \, \1_{\{|v| \leq \delta_L\}}. \]
  For any $t_*$ we have then, recalling that $t_0 = t_* / 3$,
  \begin{multline*}
    \int_{0}^{t_*} \int_0^{s} T_{t_*-s}\mathcal{L}^+T_{s-r}\mathcal{L}^+ T_r f_0 \d r \d s
    \geq
    \alpha_L^2 \, \1_{\{|v| \leq \delta_L\}}
    \int_{2 t_0}^{t_*} \int_0^{t_0}
    \frac{1}{(s-r)^d} \d r \d s
    \\
    \geq
    \frac{t_0^2}{t_*^d} \alpha_L^2 \1_{\{|v| \leq \delta_L\}}
    =
    \frac{1}{9} t_*^{2-d} \alpha_L^2 \1_{\{|v| \leq \delta_L\}}.
  \end{multline*}
  Finally, from Duhamel's formula \eqref{eq:Duhamel} we obtain
  \begin{equation*}
    f(t_*,x,v)
    \geq  \frac{1}{9} e^{-t_*} t_*^{2-d} \alpha_L^2
    \1_{\{|v| \leq \delta_L\}},
  \end{equation*}
  which gives the result.
\end{proof}

\begin{proof}[Proof of Theorem \ref{thm:main-torus} in the case of the linear relaxation Boltzmann equation]
Lemma \ref{lem:Doeblin_LBGK} allows us to apply directly Doeblin's
Theorem \ref{thm:Doeblin} to obtain fast exponential convergence to
equilibrium in the total variation distance. This rate is also explicitly
calculable. Therefore, the proof follows.
\end{proof}

\subsection{On the whole space with a confining potential}

Now we consider the equation

\begin{equation}
  \label{eq:LBGKc}
  \p_t f + v \cdot \nabla_xf - \nabla_x \Phi(x) \cdot \nabla_v f = \mathcal{L}f,
\end{equation}
where $\mathcal{L}$ is defined as in the previous section and
$x, v \in \mathbb{R}^d$.
We want to use a slightly different strategy to show the minorisation
condition based on the fact that we instantaneously produce large
velocities. We first need a result on the trajectories of particles
under the action of the potential $\Phi$. Always assuming that $\Phi$
is a $\mathcal{C}^2$ function, we consider the characteristic ordinary
differential equations associated to the transport part of
\eqref{eq:LBGKc}:
\begin{equation}
  \label{eq:ODE1}
  \begin{aligned}
    &\dot{x} = v
    \\
    &\dot{v} = - \nabla \Phi(x),
  \end{aligned}
\end{equation}
and we denote by $(X_t(x_0,v_0), V_t(x_0,v_0))$ the solution at time
$t$ to \eqref{eq:ODE1} with initial data $x(0) = x_0$, $v(0) =
v_0$. Performing time integration twice, it clearly satisfies
\begin{equation}
\label{1}
X_t(x_0, v_0)
= x_0 + v_0 t
+ \int_0^t \int_0^s \nabla \Phi(X_u(x_0, v_0)) \d u \d s
\end{equation}
for any $x_0, v_0 \in \R^d$ and any $t$ for which it is
defined. Intuitively the idea is that for small times we can
approximate $(X_t, V_t)$ by $(X^{(0)}_t, V^{(0)}_t)$ which is a
solution to the ordinary differential equation
\begin{equation}
  \label{eq:ODE2}
  \begin{aligned}
    &\dot{x} = v
    \\
    &\dot{v} = 0,
  \end{aligned}
\end{equation}
whose explicit solution is $(X_t^{(0)}, V_t^{(0)}  )= (x_0 + v_0 t, v_0)$.
If we want to hit a point $x_1$ in time $t$ then if
we travel with the trajectory $X^{(0)}$ we just need to choose
$v_0 = (x_1-x_0)/t$. Now we choose an interpolation between
$(X^{(0)}, V^{(0)})$ and $(X,V)$. We denote it by
$(X^{(\epsilon)}, V^{(\epsilon)})$ which is a solution to the ordinary
differential equation

\begin{equation}
  \label{eq:ODE3}
  \begin{aligned}
    &\dot{x} = v
    \\
    &\dot{v} = -\epsilon^2 \nabla \Phi (x),
  \end{aligned}
\end{equation}
still with initial data $(x_0, v_0)$. We calculate that
\[ X^{(\epsilon)}_t(x_0, v_0) = X_{\epsilon t}\left( x_0,
    \frac{v_0}{\epsilon} \right), \qquad V^{(\epsilon)}_t (x_0, v_0) =
  \epsilon V_{\epsilon t}\left( x_0, \frac{v_0}{\epsilon} \right). \]
Now we can see from the ODE representation (and we will make this more
precise later) that $(X,V)$ is a $C^1$ map of $(t, \epsilon, x,
v)$. Therefore if we fix $t$ and $x_0$ we can define a $C^1$ map
\[ F \: [0,1] \times \mathbb{R}^d \rightarrow \mathbb{R}^d, \] by
\[ F(\epsilon, v) = X^{(\epsilon)}_t (x_0, v). \] Then for
$\epsilon = 0$ we can find $v_*$ such that $F(0, v_*) = x_1$ as given
above. Furthermore $\nabla F(0, v_*) \neq 0$ so by the implicit
function theorem for all $\epsilon$ less than some $\epsilon_*$ we
have a $C^1$ function $v(\epsilon)$ such that
$F(\epsilon, v(\epsilon)) = x_1$. This means that
\[X_{\epsilon t} \left(x_0, \frac{v(\epsilon)}{\epsilon} \right) = x_1
  .\] So if we take $s < \epsilon_* t$ then we can choose $v$ such
that $X_s(x_0, v) = x_1$. We now need to get quantitative estimates on
$\epsilon_*$, and we do this by tracking the constants in the proof of
the contraction mapping theorem.

In order to make these ideas quantitative and to check that the
solution is in fact $C^1$ we need to get bounds on $(X_t, V_t)$ and
$\nabla \Phi (X_t)$ for $t$ is some fixed intervals. For the
potentials of interest we will have that the solutions to these ODEs
will exist for infinite time. We prove bounds on the solutions and
$\nabla \Phi (X_t)$ for any potential:

\begin{lem}
  \label{lem:Xt-bound}
  Assume that the potential $\Phi$ is $\mathcal{C}^2$ in $\R^d$. Take
  $\lambda > 1$, $R > 0$ and $x_0, v_0 \in \R^d$ with $|x_0| \leq
  R$. The solution $t \mapsto X_t(x_0,v_0)$ to \eqref{eq:ODE1} is
  defined (at least) for $|t| \leq T$, with
  \begin{equation*}
    T := \min\left\{ \frac{(\lambda-1) R}{2 |v_0|},
      \frac{\sqrt{(\lambda-1) R}}{\sqrt{2 C_{\lambda R}}}  \right\},
    \qquad
    C_{\lambda R} := \max_{|x| \leq  \lambda R} |\nabla \Phi(x)|.
  \end{equation*}
  (It is understood that any term in the above minimum is $+\infty$ if
  the denominator is $0$.) Also, it holds that
  \begin{equation*}
    |X_t(x_0, v_0)| \leq \lambda R
    \qquad \text{for $|t| \leq T$.}
  \end{equation*}
  from this we can deduce
  \begin{equation*}
    |V_t(x_0, v_0)| \leq |v_0| +C_{\lambda R} t
        \qquad \text{for $|t| \leq T$.}
  \end{equation*}
\end{lem}

\begin{proof}
  By standard ODE theory, the solution is defined in some maximal
  (open) time interval $I$ containing $0$; if this maximal interval
  has any finite endpoint $t_*$, then $X_t(x_0,v_0)$ has to blow up as
  $t$ approaches $t_*$. Hence if the statement is not satisfied, there
  must exist $t \in I$ with $|t| \leq T$ such that
  $|X_t(x_0,v_0)| \geq \lambda R$. By continuity, one may take
  $t_0 \in I$ to be the ``smallest'' time when this happens: that is,
  $|t_0| \leq T$ and
  \begin{gather*}
    X_{t_0}(x_0,v_0) = \lambda R,
    \\
    |X_{t_0}(x_0,v_0)| \leq \lambda R
    \qquad \text{for $|t| \leq |t_0|$}.
  \end{gather*}
  By \eqref{1} and using that $|t_0| \leq T$ we have
  \begin{align*}
    \lambda R = |X_{t_0}(x_0, v_0)|
    &\leq |x_0 |
    + |v_0 t_0| + \frac {t_0^2} 2 \max \{|\nabla \Phi(X_t(x_0, v_0))|
    \colon
    t \le t_0\}
    \\
    &\leq R
    + \frac{(\lambda - 1) R}{2}
      + \frac {C_{\lambda R}}{2} t_0^2
      =
      \frac{(\lambda + 1) R}{2} + \frac {C_{\lambda R}}{2} t_0^2,
  \end{align*}
  which implies that
  \begin{equation*}
    (\lambda-1) R \leq C_{\lambda R} t_0^2.
  \end{equation*}
  If $C_{\lambda R}= 0$ this is false; if $C_{ \lambda R} > 0$, then this
  contradicts with that $|t_0| \leq T$.
\end{proof}

We now follow the intuition given at the beginning of this section. However we collapse the variables $\epsilon$ and $t$ together and consequently look at $X_t\left (x, \frac{v}{t}\right )$ which is intuitively less clear but algebraically simpler.
\begin{lem}
  \label{lem:shooting}
  Assume that $\Phi \in \mathcal{C}^2(\R^d)$, and take
  $x_0, x_1 \in \R^d$. Let $R := \max\{ |x_0|, |x_1| \}$. There exists
  $0 < t_1 = t_1(R)$ such that for any $t \le t_1$ we can find a
  $|v_0| \le 4R$ such that \[X_t \left ( x_0, \frac {v_0} {t} \right)= x_1.\] In fact,
  it is enough to take $t_1 > 0$ such that
  \begin{equation*}
    C t_1^2 e^{C t_1^2 } \le \frac 1 4,
    \qquad t_1 \leq \frac{\sqrt{R}}{\sqrt{2 C_{2R}}},
    \quad \quad
    t_1\leq \frac{2 \sqrt{R}}{\sqrt{C_{9R}}},
  \end{equation*}
  where \[C := \sup_{|x| \leq 9R} |D^2 \Phi (x)| \], $C_{\lambda R}$
  is defined in Lemma \ref{lem:Xt-bound} and $D^2 \Phi$ denotes the
  Hessian matrix of $\Phi$.
\end{lem}

\begin{proof}
  We define
  \begin{gather*}
    f(t, v) = X_t \left( x_0, \frac v t \right) - x_1,
    \qquad \text{$t \neq 0$, $v \in \R^d$,}
    \\
    f(0,v) := x_0 + v - x_1,
    \qquad \text{$v \in \R^d$.}
  \end{gather*}
  Notice that due to Lemma \ref{lem:Xt-bound} with $\lambda = 9$, this
  is well-defined whenever
  \begin{equation*}
    |t| \leq \frac{2 \sqrt{R}}{\sqrt{C_{9R}}} =: t_2,
    \quad
    |v| \leq 4 R.
  \end{equation*}
  Our goal is to find a neighbourhood of $t=0$ on which there exists
  $v = v(t)$ with $f(t,v(t)) = 0$, for which we will use the implicit
  function theorem.

  Now, notice that we have
  \[f(0, x_1 - x_0) =0\]
  and
  \begin{equation*}
    \frac {\partial f} {\partial {v_i}} (0, x_1-x_0) = 1,
    \hspace{7pt} i= 1, \dots,d.
  \end{equation*}
  We can apply the implicit function theorem to find a
  neighbourhood $I$ of $t=0$ and a function $v= v(t)$ such that
  $f(t,v(t)) = 0$ for $t \in I$. However, since we need to estimate the size of
  $I$ and of $v(t)$, we carry out a constructive proof.

  Take $v_0, v_1 \in \R^d$ with $|v_0|, |v_1| \leq 4 R$, and denote
  $\tilde{v}_0 := v_0 / t$, $\tilde{v}_1 := v_1 / t$. By \eqref{1},
  for all $0 < t \leq t_2$ we have
  \begin{equation}
    \label{eq:Xtdif}
    X_t(x_0, \tilde{v}_1 )  - X_t(x_0, {\tilde{v}_0} )
    =  (\tilde v_1 -\tilde v_0) t
    + \int_0^t \int_0^s \nabla \Phi(X_u(x_0, {\tilde v_1}  ))
    - \nabla \Phi(X_u(x_0 , \tilde v_0 ))      \d u \d s.
  \end{equation}
  Take any $t_1 \leq t_2$, to be fixed later. Then Lemma
  \ref{lem:Xt-bound} implies, for all $0 \leq t \leq t_1$,
  \begin{equation*}
    |X_t(x_0, {\tilde{v}_1} )  - X_t(x_0, {\tilde{v}_0} ) |
    \le  |\tilde{v}_1 -\tilde{v}_0| t + C t_1 \int_0^t
    |X_u(x_0, {\tilde{v}_1}  ) - X_u(x_0, \tilde{v}_0 )   |   \d u.
  \end{equation*}
  by Gronwall's Lemma we have
  \begin{equation*}
    |X_t(x_0, {\tilde{v}_1} )  - X_t(x_0, {\tilde{v}_0} ) |
    \le  |\tilde{v}_1 -\tilde{v}_0| t e^{C t_1 t}
    \qquad \text{for $0 < t \leq t_1$}.
  \end{equation*}
  Using this again in \eqref{eq:Xtdif} we have
  \begin{align*}
    |X_t(x_0, {\tilde{v}_1} )  - X_t(x_0, {\tilde{v}_0} ) - (\tilde{v}_1- \tilde{v}_0)t |
    & \le |\tilde{v}_1 -\tilde{v}_0| C t_1\int_0^t  u e^{C t_1u}  \d u 
    \\
    & \le  |\tilde{v}_1 -\tilde{v}_0| t\, C t_1^2 e^{C t_1^2}.
  \end{align*}
  Taking $t_1$ such that
  \begin{equation}
    \label{eq:T1-bound0}
    C t_1^2 e^{C t_1^2 } \le \frac 1 4  
  \end{equation}
  we have
  \begin{equation*}
    |X_t(x_0, {\tilde{v}_1} ) - X_t(x_0, {\tilde{v}_0} ) - (\tilde{v}_1- \tilde{v}_0)t |
    \le \frac 1 4 |\tilde{v}_1 -\tilde{v}_0| t
  \end{equation*}
  which is the same as
  \begin{equation}
    \label{2}
    \left | X_t\left( x_0, \frac {v_1} t \right)- X_t \left( x_0, \frac {v_0} t \right) - (v_1- v_0) \right |
    \le \frac 1 4 |v_1 -v_0|,
  \end{equation}
  for any $0 < t \leq t_1$ and any $v_0, v_1$ with $|v_0|, |v_1| \leq 4 R$.
  Now, for any
  $0 \leq t \le t_1$ and $|v| \leq 4 R$ we define
  \begin{equation*}
    A_t (v) = v - f(t, v).
  \end{equation*}
  A fixed point of $A_t(v)$ satisfies $f(t, v) =0$, and by \eqref{2}
  $A_t(v)$ is contractive:
  \begin{equation*}
    |A_t(v_1) - A_t (v_0)|
    \le \frac 1 4 |v_1 -  v_0|
    \qquad \text{for $0 \leq t \leq t_1$, $|v| \leq 4 R$}.
  \end{equation*}
  (Equation \eqref{2} proves this for $0 < t \leq t_1$, and for $t= 0$
  it is obvious.) In order to use the Banach fixed-point theorem we
  still need to show that the image of $A_t$ is inside the set with
  $|v| \leq 4R$. Using \eqref{2} for $v_1 = 0$, $v_0 = v$ we also see
  that
  \begin{equation*}
    \left | X_t(x_0, 0) - X_t \left( x_0, \frac {v} t \right) + v \right |
    \le \frac 1 4 |v|,
  \end{equation*}
  which gives
  \begin{equation*}
    |A_t(v) + x_1 - X_t(x_0,0) |  \le \frac 1 4 |v|,
  \end{equation*}
  so
  \begin{equation}
    \label{eq:At_bound}
    |A_t(v) | \le \frac 1 4 |v| + |x_1| + |X_t(x_0,0) |
    \leq
    2 R + |X_t(x_0,0) |.
  \end{equation}
  If we take
  \begin{equation}
    \label{eq:T1-bound}
    t_1 \leq \frac{\sqrt{R}}{\sqrt{2 C_{2R}}}
  \end{equation}
  then Lemma \ref{lem:Xt-bound} (used for $\lambda = 2$) shows that
  \begin{equation*}
    |X_t(x_0,0) | \leq 2 R \qquad \text{for $0 \leq t \leq t_1$},
  \end{equation*}
  and from \eqref{eq:At_bound} we have
  \begin{equation*}
    |A_t(v)|
    \leq
    4 R
   \quad \text{for}\quad 0 < t \leq t_1.
  \end{equation*}
  Hence, as long as $t_1$ satisfies \eqref{eq:T1-bound0} and
  \eqref{eq:T1-bound}, $A_t$ has a fixed point $|v|$ for any
  $0 < t \leq t_1$, and this fixed point satisfies $|v| \leq 4R$.
\end{proof}

\begin{lem}
  \label{lem:HTshooting}
  Assume the potential $\Phi \in \mathcal{C}^2(\R^d)$ is bounded
  below, and let $T_s$ denote the transport semigroup associated to
  the operator
  $f\mapsto -v \cdot \nabla_x f + \nabla_x \Phi(x) \cdot \nabla_v
  f$. Given any $R > 0$ there exists a time $t_1 > 0$ such that for
  any $0 < s < t_1$ one can find constants $\alpha, R', R_2 > 0$
  (depending on $s$ and $R$) such that
  \begin{equation}
    \label{eq:HTshooting}
    \int_{B(R')} T_s (\delta_{x_0} \1_{\{|v| \leq R_2 \}} ) \d v
    \geq
    \alpha \1_{\{|x| \leq R\}},
  \end{equation}
  for any $x_0$ with $|x_0| \leq R$. The constants $\alpha, R', R_2$
  are uniformly bounded in bounded intervals of time; that is, for any
  closed interval $J \subseteq (0,t_1)$ one can find
  $\alpha, R', R_2$ for which the inequality holds for all
  $s \in J$.
\end{lem}

\begin{proof}
  Since the statement is invariant if $\Phi$ changes by an additive
  constant, we may assume that $\Phi \geq 0$ for simplicity. Using
  Lemma \ref{lem:shooting} we find $t_1$ such that for any $s < t_1$
  and every $x_1 \in B(R)$ there exists $v \in B(4R)$ (depending on
  $x_0$, $x_1$ and $s$) such that
  \[ X_s\left(x_0, \frac{v}{s} \right) = x_1. \] Since
  $v/s \in B(4R/s)$, call $R_2 := 4R/s$. We see that for every
  $x_1 \in B(0, R)$ there is at least one $u \in \R^d$ such that
  \[ (x_1, u) \in T_s\left( \{x_0\} \times \{ |v| \leq R_2\}\right). \]
  In other words,
  \begin{equation}
    \label{eq:shooting_sets}
    X_s(x_0, \{ |v| \leq R_2 \}) \supseteq B(0,R).
  \end{equation}
  This essentially contains our result, and we just need to carry out
  a technical argument to complete it and estimate the constants
  $\alpha$ and $R'$. For any compactly supported, continuous and
  positive $\varphi \: \R^d \to \R$ we have
  \begin{multline}
    \label{eq:Tsdual}
    \ird \varphi(x) \int_{B(R')} T_s (\delta_{x_0} \1_{\{|v| \leq R_2 \}})
    \d v \d x
    \\
    =
    \ird \ird \1_{\{|V_s(x,v)| < R'\}} \, \varphi(X_s(x, v)) \delta_{x_0}(x) \1_{\{|v| \leq R_2 \}})
      \d v \d x
    \\
    =
    \int_{|v| \leq R_2} \1_{\{|V_s(x_0,v)| < R'\}}\, \varphi(X_s(x_0, v) \d v,
  \end{multline}
  since the characteristics map $(x,v) \mapsto (X_s(x,v), V_s(x,v))$
  is measure-preserving. If we write the energy as
  $H(x,v) = |v|^2/2 + \Phi(x)$ and call
  \begin{equation*}
    E_0 := \sup \, \{H(x,v) \,:\, |x| < R, |v| < R_2\}.
  \end{equation*}
  Then for all $s \geq 0$
  \begin{equation*}
    H(X_s(x_0,v), V_s(x_0,v)) \leq E_0,
  \end{equation*}
  and in particular
  \begin{equation*}
    |V_s(x_0,v)| \leq \sqrt{2E_0}.
  \end{equation*}
  If we take $R' > \sqrt{2 E_0}$ then the term
  $\1_{\{|V_s(x_0,v)| < R'\}}$ is always $1$ in \eqref{eq:Tsdual} and we
  get
  \begin{equation*}
    \ird \varphi(x) \int_{B(R')} T_s (\delta_{x_0} \1_{\{|v| \leq R_2 \}})
    \d v \d x
    =
    \int_{|v| \leq R_2} \varphi(X_s(x_0, v)) \d v.
  \end{equation*}
  Now, take an $M > 0$ such that
  $|\Jac_v X_s(x, v) | \leq M$ for all $(x,v)$ with $|x| \leq R$ and
  $|v| \leq R_2$. (Notice this $M$ depends only on $\Phi$, $R$ and
  $R_2$.) Then
  \begin{align*}
    \int_{|v| \leq R_2} \varphi(X_s(x_0, v)) \d v
    &\geq
    \frac{1}{M} \int_{|v| \leq R_2} \varphi(X_s(x_0, v))
    |\Jac_v X_s(x_0, v) | \d v
    \\
    &=
      \frac{1}{M} \int_{X_s(x_0, \{|v| \leq R_2\})} \varphi(x) \d x
      \geq
      \frac{1}{M} \int_{B(0,4R)} \varphi(x) \d x,
  \end{align*}
  where we have used \eqref{eq:shooting_sets} in the last step. In sum
  we find that
  \begin{equation*}
    \ird \varphi(x) \ird T_s (\delta_{x_0} \1_{\{|v| \leq R_2 \}})
    \d v \d x
    \geq
    \frac{1}{M} \int_{B(0,R)} \varphi(x) \d x
  \end{equation*}
  for all compactly supported, continuous and positive functions
  $\varphi$. This directly implies the result.
\end{proof}

\begin{lem}[Doeblin condition for linear relaxation Boltzmann equation with a confining
  potential]
  \label{lem:LBEdoeblin}
  Let the potential $\Phi \: \R^d \to \R$ be a $\mathcal{C}^2$
  function with compact level sets. Given $t > 0$ and $K > 0$ there
  exist constants $\alpha, \delta_X, \delta_V > 0$ such that any
  solution $f$ to equation \eqref{eq:LBGKc} with initial condition
  $f_0 \in \P(\R^d \times \R^d)$ supported
  on $B(0,K) \times B(0,K)$ satisfies
  \begin{equation*}
    f(t, x, v) \geq \alpha \1_{\{|x| < \delta_X\}} \ \1_{\{|v| < \delta_V\}}
  \end{equation*}
 in the sense of measures.
\end{lem}

\begin{proof}
  Fix any $t, K > 0$. Set
  \[H_{\max}(K)=\max \, \left \{ H(x,v)= |v|^2/2+\Phi(x) \,:\, x \in B(0,K), v \in B(0,K) \right \} \] and then define
  \[R:=\max\, \left \{ |x| \,:\, \Phi(x) \leq H_{\max}(K)\right \}. \] Since our
  conditions on $\Phi$ imply that its level sets are compact we know
  that $R$ is finite. We use Lemma \ref{lem:HTshooting} to find
  constants $\alpha, R_2 > 0$ and an interval $[a,b] \subseteq (0,t)$
  such that
  \begin{equation*}
    \ird T_s (\delta_{x_0} \1_{\{|v| \leq R_2 \}} ) \d v
    \geq
    \alpha \1_{\{|x| \leq R\}},
  \end{equation*}
  for any $x_0$ with $|x_0| \leq R$ and any $s \in [a,b]$. From Lemma
  \ref{lem:HL} we will use that there exists a constant $\alpha_L > 0$
  such that
  \begin{equation}
    \label{eq:HL2}
    \mathcal{L}^+ g(x,v) \geq
    \alpha_L \left( \ird g(x,u) \d u \right) \1_{\{|v| \leq R_2\}}
  \end{equation}
  for all nonnegative measures $g$. We first notice that we can do the
  same estimate as in formula \eqref{eq:Duhamel}, where now $(T_t)_{t
    \geq 0}$ represents the semigroup generated by the operator $-v
  \cdot \nabla_x f + \nabla_x \Phi(x) \cdot \nabla_v f$:
  \begin{equation}
    \label{eq:Duhamel2}
    e^t f_t \geq \int_0^t \int_0^s T_{t-s} \mathcal{L}^+ T_{s-r} \mathcal{L}^+ T_r f_0
    \d r \d s.
  \end{equation}
  Take $x_0, v_0 \in B(0,K)$, and call $f_0 :=
  \delta_{(x_0,v_0)}$. For all $r$ we have by the definition of $R$
  that
  \begin{equation}
    \label{eq:Xr-bound}
    |X_r(x_0,v_0)| \leq  R
    \qquad
    \text{for all $0 \leq r$.}  
  \end{equation}
  For any $r > 0$, since
  $T_r f_0 = \delta_{(X_r(x_0,v_0), V_r(x_0, v_0))}$, using
  \eqref{eq:HL} gives
  \begin{equation*}
    \mathcal{L}^+ T_r f_0 \geq
    \alpha_L  \delta_{X_r(x_0,v_0)}(x) \1_{\{|v| \leq R_2\}}.
  \end{equation*}
  Then, using \eqref{eq:Xr-bound} and our two lemmas, whenever
  $s-r \in [a,b]$ we have
  \begin{align*}
    \mathcal{L}^+ T_{s-r} \mathcal{L}^+ T_r f_0
    &\geq
    \alpha_L \left( \ird T_{s-r} \mathcal{L}^+ T_r f_0 \d u \right)
      \1_{\{|v| \leq R_2\}}
    \\
    &
      \geq
      \alpha_L^2 \left(
      \ird T_{s-r} \Big( \delta_{X_r(x_0,v_0)}(x)
      \1_{\{|u| \leq R_2\}}
      \Big)
      \d u
      \right)
      \1_{\{|v| \leq R_2\}}
    \\
    &
      \geq
      \alpha_L^2 \alpha\,
      \1_{\{|x| \leq R\}}
      \1_{\{|v| \leq R_2\}}.
  \end{align*}
  We now need to allow for a final bit of movement along the flow
  $T_{t-s}$. The time gradient of the flow is bounded by
  $|\nabla_x(\Phi(X_t(x,v))|+|V_t(x,v)|$ and this quantity is bounded
  on sublevel sets of the Hamiltonian which are preserved by the flow
  so there exists a sufficiently small, quantifiable $\epsilon>0$ so
  that for all $0 \leq \tau \leq \epsilon$ we have
  \begin{equation}
    \label{eq:p1}
    T_{\tau} \Big( \1_{B(R)}(x) \1_{B(R_2)}(v) \Big)
    \geq
    \1_{B(R/2)}(x) \1_{B(R_2/2)}(v).
  \end{equation}
  (We point out a way to quantify $\epsilon$: since $T_\tau$ is
  measure-preserving, we have $T_\tau (h)(x,v) = h(X_\tau(x,v),
  V_t(x,v))$ for any function $h = h(x,v)$. Define the inverse flow of $T_\tau$ by $G_\tau$, if we denote $G_\tau(x_0 ,v_0 ) =(Y_\tau(x_0, v_0), Z_\tau(x_0, v_0))$, then $Y_\tau, Z_\tau$ satisfies
\begin{equation}
  \nonumber
  \begin{aligned}
    &\dot{y} = -z
    \\
    &\dot{z} = W(y),
  \end{aligned}
\end{equation}
 with initial condition $\{x_0, y_0 \}$. Hence \eqref{eq:p1} holds
   if $|Y_\tau(x,v)| \leq R$ and $|Z_\tau(x,v)| \leq R_2$ for all $|x|
   \leq R/2$, $|v| \leq R_2/2$. It's easily seen that the result of Lemma \ref{lem:Xt-bound} will still hold for $G_\tau$, so we can take  
   \begin{equation*}
    \epsilon = \min\left\{ \frac{R}{2 R_2},
      \frac{\sqrt{ R}}{2 \sqrt{ C_{ R}}}   ,  \frac{R_2}{2 R} \right\},
    \qquad
    C_{ R} := \max_{|x| \leq R} |\nabla \Phi(x)|.
  \end{equation*}
  by Lemma \ref{lem:Xt-bound}.)
   
  From \eqref{eq:p1}, for all $t, s, r$ such that $t-s \leq \epsilon$
  and $s-r \in (a, b)$ we have
  \[  T_{t-s}\mathcal{L}^+ T_{s-r} \mathcal{L}^+ T_r f_0 \geq
    \alpha_L^2 \alpha\,
    \1_{\{|x| \leq R/2\}} \, \1_{\{|v| \leq R_2/2\}}. \]
  We have then
  \begin{multline*}
    \int_{0}^{t} \int_0^{s}
    T_{t-s}\mathcal{L}^+T_{s-r}\mathcal{L}^+ T_r f_0 \d r \d s
    \geq
    \alpha_L^2 \alpha
    \int_{t-\epsilon}^{t} \int_{s-b}^{s-a}
    \1_{\{|x| \leq R/2\}} \1_{\{|v| \leq R_2/2\}} \d r \d s
    \\
    =
    \alpha_L^2 \alpha \epsilon (b-a)
    \1_{\{|x| \leq R/2\}} \1_{\{|v| \leq R_2/2\}}.
  \end{multline*}
  Finally, from Duhamel's formula \eqref{eq:Duhamel2} we obtain
  \begin{equation*}
    f(t,x,v) \geq e^{-t} \alpha_L^2 \alpha \epsilon (b-a)
    \, \1_{\{|x| \leq R/2\}} \1_{\{|v| \leq R_2/2\}},
  \end{equation*}
  which gives the result.  
\end{proof}

\begin{lem}[Lyapunov condition]\label{lem:LBEconfining}
  Suppose that $\Phi(x)$ is a $\mathcal{C}^2$ function satisfying
  \[ x \cdot \nabla \Phi(x) \geq \gamma_1 |x|^2 + \gamma_2 \Phi(x) - A \] for
  positive constants $A$, $\gamma_1 $ $\gamma_2$.  Then we have that
  \[ V(x,v)= 1+\Phi(x) + \frac{1}{2}|v|^2 +
    \frac{1}{4}x \cdot v + \frac{1}{8}|x|^2 
  \] is a function for which the semigroup satisfies Hypothesis
  \ref{confinement}.
\end{lem}
\begin{remark}
If $\Phi$
  is superquadratic at infinity (which is implied by earlier
  assumptions) then $V$ is equivalent to $1+H(x,v)$ where the energy is defined as $H(x,v) = |v|^2/2+\Phi(x)$. So the total variation distance weighted by $V$ is equivalent to the total variation distance weighted by $1+H(x,v)$.
\end{remark}

\begin{proof}

  We look at the forwards operator acting on an observable $\phi$,
  \[ \UU \phi = v \cdot \nabla_x \phi - \nabla_x \Phi(x) \cdot
    \nabla_v \phi + \LL^* \phi
    =: \TT^* \phi + \LL^* \phi, \]
  where $\LL^*$ is the adjoint of the linear relaxation Boltzmann
  operator $\LL$, given by
  \[ \LL^* \phi(x,v) = \int \phi(x,u) \M(u) \d u  - \phi(x,v). \]
  We want a function $V(x,v)$ such that
  \[ \UU V \leq - \lambda V + K \] for some constants
  $\lambda > 0, K\geq 0$. We need to make the assumption that
  \begin{equation}
    \label{eq:Phi-bound}
    x \cdot \nabla_x \Phi(x) \geq \gamma_1 |x|^2 + \gamma_2 \Phi(x) - A.
  \end{equation}
  for some positive constant $A, \gamma_1, \gamma_2$. We then try the
  function
  \[ V(x,v) = H(x,v) + a x \cdot v + b |x|^2
    =  \Phi(x) + \frac{1}{2}|v|^2 + a x \cdot v + b |x|^2, \]
 with $a, b>0$ to be fixed later. We want this to be positive so we impose $a^2 < 2b$.  Using that
  \begin{equation*}
    \LL^* (|v|^2) = d - |v|^2,
    \qquad
    \LL^* (x \cdot v) = -x \cdot v,
    \qquad
    \LL^* (\Phi(x)) = \LL^* (|x|^2) = 0
  \end{equation*}
  and that
  \begin{equation*}
    \TT^* (H(x,v)) = 0,
    \qquad
    \TT^* (x \cdot v) = |v|^2 - x \cdot \nabla_x \Phi(x),
    \qquad
    \TT^* (|x|^2) = 2 x \cdot v,
  \end{equation*}
  we see that
  \begin{align*}
    \UU (V) =& \frac{d}{2} - \frac{1}{2}|v|^2 - a x \cdot v + a |v|^2 - a x \cdot \nabla_x \Phi(x) +2b x \cdot v\\
    \leq & C' -\left(\frac{1}{2} - a \right) |v|^2 +(2b-a) x \cdot v - a
           \gamma_1 |x|^2 - a \gamma_2\Phi(x),
  \end{align*}
  where we have used \eqref{eq:Phi-bound}, and $C' := \frac{d}{2} +
  aA$. Now, taking $a=1/4, b=1/8$,
  \begin{align*}
    \UU (V)
    = & C' - \frac{1}{4} |v|^2
        - \frac{\gamma_1}{4} |x|^2 - \frac{\gamma_2}{4}\Phi(x) 
    \\
    \leq & C' - \frac{\min(\gamma_1, 1)}{4} (|x|^2 + |v|^2) -
        \frac{\gamma_2}{4} \Phi(x)
    \\
    \leq
      & C' -\frac{\min(\gamma_1, 1)}{4}
        \left( \frac{1}{2}|v|^2 + \frac{1}{4}x \cdot v + \frac{1}{8}
        |x|^2 \right) - \frac{\gamma_2}{4} \Phi(x).
  \end{align*}
  So $V(x,v)$ works with
  \[ \lambda = \frac{\min(\gamma_1, \gamma_2, 1)}{4}. \qedhere\]
\end{proof}

\medskip
\begin{proof}[Proof of Theorem \ref{thm:main-confining} in the case of
  the linear relaxation Boltzmann equation]
  The proof follows by applying Harris's Theorem since Lemmas
  \ref{lem:LBEdoeblin} and \ref{lem:LBEconfining} show that the
  equation satisfies the hypotheses of the theorem.
\end{proof}

\subsection{Subgeometric convergence}

When we do not have the superquadratic behaviour of the confining
potential at infinity we can still use a Harris type theorem to show
convergence to equilibrium. This time we must pay the price of having
subgeometric rates of convergence. We use the subgeometric Harris's Theorem given in Section \ref{sec:Harris} which
can be found in Section 4 of \cite{H10}.
Now instead of our earlier assumption on the confining potential $\Phi$, we instead make a weaker assumption that $\Phi$ is a $C^2$ function satisfying

\[ x \cdot \nabla_x \Phi(x) \geq \gamma_1 \langle x \rangle^{2\beta} + \gamma_2 \Phi(x) - A, \]  for some positive constant $A, \gamma_1, \gamma_2$, where
\[ \langle x \rangle = \sqrt{1+|x|^2}, \] and $ \beta \in (0,1)$.

\begin{proof}[Proof of Theorem \ref{thm:main-subgeometric} in the case of the linear relaxation Boltzmann equation]
We have already proved the minorisation condition. We can also replicate the calculations for the Lyapunov function to get that in this new situation, take the $V$ in Lemma \ref{lem:LBEconfining}, we have for $a=1/4, b=1/8$ that
\[ \UU V \leq C' - \frac{1}{4}|v|^2 - \frac{\gamma_1}{4} \langle x \rangle^{2\beta} - \frac{\gamma_2}{4} \Phi(x). \] For $x, y \geq 1$
\[ (x+y)^\beta \leq x^\beta + y^\beta.  \]
So we have
\begin{align*}
\UU V \leq & C' - \frac{\min(\gamma_1, 1)}{4} \left(\langle v \rangle^2 + \langle x \rangle^{2\beta} \right ) - \frac{\gamma_2}{4} \Phi(x)\\
\leq  &C'' - \frac{\min(\gamma_1, 1)}{4} \left ( 1+|x|^2+|v|^2 \right )^\beta - \frac{\gamma_2}{4} \Phi(x)^\beta \\
\leq & C'' - \lambda\left(1+ \frac{1}{2}|v|^2 + \frac{1}{4}x \cdot v + \frac{1}{8} |x|^2 \right)^\beta - \lambda \Phi(x)^\beta\\
\leq & C'' - \lambda \left( \Phi(x) +\frac{1}{2}|v|^2 + \frac{1}{4}x \cdot v + \frac{1}{8} |x|^2 \right)^\beta,
\end{align*} for some constant $\lambda, C''>0$ that can be explicitly computed, so we have that
\[ \UU V \leq - \lambda V^\beta + C''. \] This means we can take $\phi(s) = 1+s^\beta$. Therefore, for $u$ large
\[ H_\phi(u) = \int_1^u \frac{1}{1+t^\beta} \mathrm{d}t \sim 1+ u^{1-\beta}, \] and for $t$ large
\[ H_{\phi}^{-1}(t)  \sim 1+t^{1/(1-\beta)} \]  and
\[ \phi \circ H_\phi^{-1}(t) \sim (1+t)^{\beta/(1-\beta)}. \]
\end{proof}

\section{The linear Boltzmann Equation}

We now look at the linear Boltzmann equation. This has been studied in
the spatially homogeneous case in \cite{BCL15, CEL17}. Here the
interest is partly that this is a more complex and physically relevant
operator. Also, it presents less globally uniform behaviour in $v$
which means that we have to use a Lyapunov function even on the
torus. Apart from this, the strategy is very similar to that from the
linear relaxation Boltzmann equation. The full Boltzmann equation has
been studied as a Markov process in \cite{FM01}, the linear case is
similar and more simple. It is well known that this equation preserves 
positivity and mass, which follows from standard techniques
both in the spatially homogeneous case and the case with transport. The
Lyapunov condition on the torus and the bound below on the jump
operator have to be verified in this situation.

We consider for $x \in \mathbb{T}^d$
\begin{equation}
\label{LBEharris2}
\partial_t f + v \cdot \nabla_x f = \int_{\mathbb{R}^d}\int_{\mathbb{S}^{d-1}} B \left( \frac{v-v_*}{|v-v_*|}\cdot \sigma, |v-v_*|\right) \left(f(v')\mathcal{M}(v_*')-f(v)\mathcal{M}(v_*) \right) \mathrm{d}\sigma \mathrm{d}v_*.
\end{equation} 
We assume that $B$ splits as 
\begin{equation} \label{eq:B-splits}
B \left( \frac{v-v_*}{|v-v_*|}\cdot \sigma, |v-v_*|\right) = b
\left( \frac{v-v_*}{|v-v_*|} \cdot \sigma \right)
|v-v_*|^{\gamma}.
\end{equation}
We make a cutoff assumption that $b$ is
integrable in $\sigma$. In fact, we make a much stronger assumption
that $b$ is bounded below by a constant. We also work in the hard
spheres/Maxwell molecules regime that is to suppose $\gamma \geq
0$. When working with the Boltzmann collision kernel we have a choice
of parametrizations for the incoming velocities. Choosing this
`$\sigma$-parametrization' allows for simpler calculations, but is not
otherwise essential. We notice that for the physical hard spheres
kernel the angular kernel $b$ is bounded below in the $\sigma$
parametrization. We have
\[\partial_t f + v \cdot \nabla_x f = \mathcal{L}^+ f
  - \kappa(v) f,  \]
where
$\kappa(v) \geq 0$ and $\kappa(v)$ behaves like $|v|^{\gamma}$
for large $v$; that is,
\begin{equation}
  \label{eq:sigma_bound}
  0 \leq \kappa(v) \leq (1 + |v|^2)^{\gamma/2},
  \qquad v \in \R^d.
\end{equation}
See \cite{CEL17} Lemma 2.1 for example.

We also look at the situation where the spatial variable is in
$\mathbb{R}^d$ and we have a confining potential. With hard
sphere, the operator $\mathcal{L}^+$ acting on $x\cdot v$ produces error terms
which are difficult to deal with. We show that when we have hard
spheres with $\gamma >0$ we can still show exponential convergence
when $\Phi(x)$ is growing at least as fast as $|x|^{\gamma+2}$. In the
subgeometric case we suppose $\Phi(x)$ grows at least as fast as $|x|^{\epsilon+1} ,\epsilon >0$. The equation is
\begin{equation}
\label{LBEharris3}
\partial_t f + v \cdot \nabla_x f - \left( \nabla_x \Phi(x) \cdot \nabla_v f\right) = Q(f,\mathcal{M}).
\end{equation}

We begin by proving lemmas which are useful for proving the Doeblin
condition in both situations. We want to reduce to a similar situation
to the linear relaxation Boltzmann equation.

\begin{lem}
  \label{lem:reducing}
  Let $f$ be a solution to \eqref{LBEharris2} or \eqref{LBEharris3}, and
  define $H(x,v) := |v|^2/2$ on the torus for \eqref{LBEharris2} or
  $H(x,v) := \Phi(x) + |v|^2/2$ in the whole space for
  \eqref{LBEharris3}, where $\Phi$ is a $\mathcal{C}^2$ potential
  bounded below. Take $E_0 > 0$ and assume that $f$ has initial
  condition $f_0 = \delta_{(x_0,v_0)}$ with
  \begin{equation*}
    H(x_0,v_0) \leq E_0.
  \end{equation*}
  Then there exists a constant $C_1 > 0$ such that
  \[
    f(t,x,v) \geq e^{-tC_1} \int_0^t
    \int_0^s T_{t-s} \widetilde{\mathcal{L}}^+ T_{s-r} \widetilde{\mathcal{L}}^+
    T_r (\1_{E} f_0(x,v))   \d r \d s,
  \]
  where
  \begin{equation*}
    \widetilde{\mathcal{L}}^+ g := \1_E \mathcal{\mathcal{L}}^+ g,
    \qquad
    E := \{(x,v) \in \R^d \times \R^d \,:\, H(x,v) \leq E_0\}.
  \end{equation*}
\end{lem}

\begin{proof}
  Call $(X_t(x,v), V_t(x,v))$ the solution to the backward
  characteristic equations obtained from the transport part of either
  \eqref{LBEharris2} or \eqref{LBEharris3}.
  Let us call
  \[ \Sigma(s,t,x,v) = e^{\int_s^t \kappa(V_r(x,v)) \mathrm{d}r}.
  \]
  Looking at Duhamel's formula again we get
  \begin{equation*}
    f(t,x,v)
    = \Sigma(0,t,x,v) T_{t} f_0
    + \int_0^t \Sigma(0,t-s,x,v) (T_{t-s} \mathcal{L}^+ f_s)(x,v) \d s
  \end{equation*}
  If a function $g = g(x,v)$ has support on the set
  $$E := \{(x,v) \,:\, H(x,v) \leq E_0 \},$$ then the same is true of
  $T_t g$ (since the transport part preserves energy). On the set $E$
  we have, using \eqref{eq:sigma_bound},
  \[
    \int_s^t \kappa(V_r(x,v)) \d r
    \leq (t-s)C \left( 1 + 2 E_0 \right)^{\gamma/2}
    =: (t-s)C_1,
    \qquad (x,v) \in E.
  \]
  Hence
  \begin{align*}
    f(t,x,v)
    &\geq \Sigma(0,t,x,v) T_{t} (\1_E f_0)
    + \int_0^t \Sigma(0,t-s,x,v) (T_{t-s} (\1_{E} \mathcal{L}^+ f_s))(x,v) \d s
    \\
    &\geq e^{-t C_1} T_{t} (\1_E f_0)
    + \int_0^t e^{-(t-s) C_1} (T_{t-s} (\1_{E} \mathcal{L}^+ f_s))(x,v) \d s
    \\
    &= e^{-t C_1} T_{t} f_0
    + \int_0^t e^{-(t-s) C_1} (T_{t-s} (\widetilde{\mathcal{L}}^+ f_s))(x,v) \d s,
  \end{align*}
  where we define
  \begin{equation*}
    \widetilde{\mathcal{L}}^+ g := \1_E \mathcal{L}^+ g.
  \end{equation*}
  Iterating this formula we obtain the result.
\end{proof}

The next lemma we want to prove is a local version of Lemma
\ref{lem:HL}, which states that the operator $\mathcal{L}^+$ allows
jumps between any two velocities with a probability which is bounded
below, provided the size of the two velocities is bounded by a fixed
number. In order to do this, let us first rewrite the operator
$\mathcal{L}^+$. We have that
\[ \mathcal{L}^+f = \int_{\mathbb{R}^d} \int_{\mathbb{S}^{d-1}}
  b\left(\frac{v-v_*}{|v-v_*|}\cdot \sigma \right)|v-v_*|^{\gamma}
  f(v') \mathcal{M}(v_*') \mathrm{d}\sigma \mathrm{d}v_*.
\]
Using the Carleman representation we rewrite this as
\[
  \mathcal{L}^+ f = \int_{\mathbb{R}^d} \frac{f(v')}{|v-v'|^{d-1}}
  \int_{E_{(v,v')}} B(|u|, \xi)\mathcal{M}(v_*')\mathrm{d}v_*' \d v',
\]
where $E_{(v,v')}$ denotes the hyperplane
$\{ v'_* \in \R^d \mid (v - v') \cdot (v - v'_*) = 0 \}$, and the
integral in $v_*'$ is understood to be with respect to
$(n-1)$-dimensional measure on this hyperplane. We want to bound this in
the manner of Lemma \ref{lem:HL} from the first part.  We look at hard
spheres and no angular dependence, which means
\[
  B(|u|, \xi) = C|u|^{\gamma} \xi^{d-2}
\]
with $\gamma \geq 0$. We also have that
\[ \xi = \frac{|v-v'|}{|2v-v'-v'_*|}, \hspace{10pt} |u|=
  |2v-v'-v'_*|.
\]
So we have that
\[
  \mathcal{L}^+f = \int_{\mathbb{R}^d} \frac{f(v')}{|v-v'|}
  \int_{E_{(v,v')}}|2v-v'-v_*'|^{\gamma -d -2} \mathcal{M}(v_*')
  \mathrm{d}v_*' \d v'.
\]
With this we can give the following lower bound of $\mathcal{L}^+$,
which the reader can compare to Lemma \ref{lem:HL}:

\begin{lem}
  \label{lem:L-Boltzmann}
  Consider the positive part $\mathcal{L}^+$ of the linear Boltzmann operator
  for hard spheres, assuming \eqref{eq:Bsplits} with $\gamma \geq 0$,
  and \eqref{eq:b-hyp}. For all $R_L, r_L > 0$, there exists
  $\alpha > 0$ such that for all $g \in \mathcal{P}$
  \[
    \mathcal{L}^+ g (v) \geq
    \alpha  \int_{B(R_L)} g(u) \d u
    \qquad \text{for all $v \in \R^d$ with $|v| \leq r_L$}.
  \]
\end{lem}

\begin{proof} 
  First we note that on $E_{(v,v')}$ we have
  \[ |2v -v' -v_*'|^{-d-2} \geq C_d \exp \left( - \frac{1}{2}|v-v_*'|^2 - \frac{1}{2}|v-v'|^2 \right). \] Then since $\gamma \geq 0$ we have
  \[ |2v-v'-v_*'|^\gamma = \left(|v-v'|^2 + |v-v'_*|^2 \right)^{\gamma/2} \geq |v-v'_*|^\gamma.  \] So this means that
  \begin{align*}
    \int_{E_{(v,v')}} |2v-v'-&v'_*|^{\gamma -d -2} \mathcal{M}(v_*')
    \mathrm{d}v_*'
    \\&\geq C e^{-|v-v'|^2/2} \int_{E_{(v,v')}}
                          |v-v_*'|^\gamma\exp \left(-
                          \frac{1}{2}|v-v'_*|^2 - \frac{1}{2}|v_*'|^2
                          \right)\mathrm{d}v_*'
    \\
    &\geq C e^{-|v-v'|^2/2 - |v|^2/2}\int_{E_{(v,v')}}|v-v_*'|^\gamma e^{-|v-v'_*|^2} \mathrm{d}v_*'\\
    &=  C' e^{-|v-v'|^2/2 - |v|^2/2}.
  \end{align*}
  So we have that
  \begin{align*}
    \mathcal{L}^+ f (v) & \geq C \int_{\mathbb{R}^d} f(v') |v-v'|^{-1}
                 e^{-|v-v'|^2/2 - |v|^2/2}\mathrm{d}v'
    \\
    & \geq C \int_{\mathbb{R}^d} f(v') e^{-2|v'|^2-3|v|^2} \d v'\\
    & \geq C e^{-2R_L^2} e^{-3|v|^2} \int_{B(0,R_L)} f(v') \mathrm{d}v',
  \end{align*}
  which is a similar bound to the one we found in Lemma
  \ref{lem:HL}. This gives the result by choosing
  $\alpha := C \exp({-2R_L^2 - 3|r_L|^2})$.
\end{proof}

\subsection{On the torus}

Now we work specifically on the torus. For the minorisation we can
argue almost exactly as for the linear relaxation Boltzmann equation.

\begin{lem}[Doeblin condition]
  \label{lem:HL2}
  Assume \eqref{eq:Bsplits} with $\gamma \geq 0$, and
  \eqref{eq:b-hyp}. Given $t_* > 0$ and $R > 0$ there exist constants
  $0 < \alpha < 1$, $\delta_L > 0$ such that any solution
  $f = f(t,x,v)$ to the linear Boltzmann equation \eqref{LBEharris2}
  on the torus with initial condition $f_0 = \delta_{(x_0,v_0)}$ with
  $|v_0| \leq R$ satisfies
  \[ f(t_*,x,v) \geq \alpha \1_{\{|v| \leq \delta_L\} } \]
in the sense of measures.
\end{lem}

\begin{proof}
  Take $f_0 := \delta_{(x_0,v_0)}$, where
  $(x_0,v_0) \in \T^d \times \R^d$ is an arbitrary point with
  $|v_0| \leq R$.  From Lemma \ref{lem:HT} (with $R > \sqrt{d}$ and
  $t_0 := t_*/3$) we will use that there exist $\delta_L, R' > 0$ such
  that
  \begin{equation}
    \label{eq:BT1}
    \int_{B(R')} T_{t} \Big(\delta_{x_0}(x) \1_{\{|v| \leq \delta_L\}}\Big) \d v
    \geq
    \frac{1}{t^d}
    \qquad \text{for all $x_0 \in \T^d$, $t > t_0$}.
  \end{equation}
  Also, Lemma \ref{lem:L-Boltzmann} gives an $\alpha > 0$ such that
  \begin{equation}
    \label{eq:BT2}
    \mathcal{L}^+ g \geq \alpha
    \left( \int_{B(R_L)} g(x,u) \d u\right) \, \1_{\{|v| \leq \delta_L\}},
  \end{equation}
  where $R_L := \max\{R', R\}$.  Finally, from Lemma
  \ref{lem:reducing} we can find $C_1 > 0$ (depending on $R$) such
  that
  \begin{equation*}
    f(t,x,v) \geq e^{-t C_1} \int_0^t
    \int_0^s T_{t-s} \widetilde{\mathcal{L}}^+ T_{s-r} \widetilde{\mathcal{L}}^+
    T_r (\1_{E} \delta_{(x_0,v_0)})   \d r \d s,
  \end{equation*}
  where $E$ is the set of points with energy less than $E_0$, with
  \begin{equation*}
    E_0 := \max\{ R^2/2, \delta_L^2/2\},
  \end{equation*}
  and we recall that $\widetilde{\mathcal{L}}^+ f := \1_E \mathcal{L}^+ f$. Due to our
  choice of $E_0$, we see that equation \eqref{eq:BT1} also holds with
  $\widetilde{\mathcal{L}}^+$ in the place of $\mathcal{L}^+$. One can then carry out the
  same proof as in Lemma \ref{lem:Doeblin_LBGK}, using estimates
  \eqref{eq:BT1} and \eqref{eq:BT2} instead of the corresponding ones
  there.
\end{proof}

Since our Doeblin condition holds only on sets which are bounded in
$|v|$, we do need a Lyapunov functional in this case (as opposed to
the linear relaxation Boltzmann equation, where Lemma \ref{lem:Doeblin_LBGK} gives a
lower bound for all starting conditions $(x,v)$). Testing with $V=v^2$
involves proving a result similar to the moment control result from
\cite{BCL15}. Instead of the $\sigma$ representation we use the
$n$-representation for the collisions:
\begin{equation*}
  v' = v - n (u \cdot n),
  \qquad
  v'_* = v_* + n (u \cdot n).
\end{equation*}
By our earlier assumption, the collision kernel can be written as
\begin{equation*}
  \tilde{B}(|v-v_*|, |\xi|) = |v-v_*|^\gamma \tilde{b}(|\xi|),
\end{equation*}
where
\begin{equation*}
  \xi := \frac{u \cdot n}{|u|},
  \qquad
  u := v-v_*.
\end{equation*}
Here the $\tilde{B},\tilde{b}$ are different from those in the $\sigma$
representation because of the change of variables.
We also have by assumption that $\tilde{b}$ is normalised, that is,
\begin{equation*}
  \int_{\mathbb{S}^d} \tilde{b}(|w \cdot n|) \d n = 1
\end{equation*}
for all unit vectors $w \in \S^{d-1}$.

\begin{lem} \label{lem:HT-Boltzmann}
The function $V(x,v) = |v|^2$ is a Lyapunov function for the linear Boltzmann equation on the torus in the sense that it is a function for which the associated semigroup satisfies Hypothesis \ref{confinement}.
\end{lem}

\begin{proof} We define $\mathcal{L}$ to be the linear Boltzmann operator.
  Using the weak formulation of the operator,
  \begin{equation*}
    \int_{\mathbb{R}^d} \mathcal{L}(f) |v|^2 \d v
    =
    \int_{\mathbb{R}^d} \int_{\mathbb{R}^d} \int_{\mathbb{S}^{d-1}} f(v) \M(v_*) |v-v_*|^\gamma \tilde{b}(|\xi|)
    (|v'|^2 - |v|^2) \d n \d v \d v_*.
  \end{equation*}
  In other words,
  \begin{equation*}
    \LL^* (|v|^2)
    =
    \int_{\mathbb{R}^d} \int_{\mathbb{S}^{d-1}} \M(v_*) |v-v_*|^\gamma \tilde{b}(|\xi|)
    (|v'|^2 - |v|^2) \d n \d v_*.
  \end{equation*}
  We are going to prove the Lyapunov condition by showing that
  \begin{equation*}
    \int_{\R^d} \int_{\R^d} (\LL(f) + \TT(f)) |v|^2 \d x \d v
    \leq
    -\lambda \int_{\R^d} \int_{\R^d} f |v|^2  \d x \d v
    + K \int_{\R^d} \int_{\R^d} f  \d x \d v,
  \end{equation*}
  where $\TT f = -v \nabla_x f$ is the transport operator. The transport part plays no role, since
  \begin{equation*}
    \int_{\R^d} \int_{\R^d} \TT(f) |v|^2  \d x \d v = 0.
  \end{equation*}
  For the collisional part, we notice that
  \begin{align*}
    |v'|^2 - |v|^2
    &= |v_*|^2 - |v'_*|^2
    = - (u\cdot n)^2 - 2 (v_* \cdot n) (u \cdot n)
      \\
    &= - |u|^2 \xi^2 - 2 (v_* \cdot n) (v \cdot n)
      + 2 (v_* \cdot n)^2
      \\
    &= - |v|^2 \xi^2 - |v_*|^2 \xi^2
      + 2v \cdot v_* \xi^2
      - 2 (v_* \cdot n) (v \cdot n)
    + 2 (v_* \cdot n)^2.
  \end{align*}
  Note that the first term is negative and quadratic in $v$, and the
  rest of the terms are of lower order in $v$. Hence, calling
    $\gamma_b := \int_{\mathbb{S}^{d-1}} \xi^2 \tilde{b}(|\xi|) \d \xi $
  we have
  \begin{align*}
    \int_{\mathbb{R}^d} \mathcal{L}(f) |v|^2 \d v
    =
    &- \gamma_b \int_{\mathbb{R}^d} |v|^2 f(v) \int_{\mathbb{R}^d} \M(v_*) |v-v_*|^\gamma \d v_* \d v
    \\
    &- \gamma_b \ird f(v) \int_{\mathbb{R}^d} |v_*|^2 \M(v_*) |v-v_*|^\gamma \d v_* \d v
    \\
    &+ 2\gamma_b \int_{\mathbb{R}^d} v f(v) \int_{\mathbb{R}^d} v_* \M(v_*) |v-v_*|^\gamma \d v_* \d v
    \\
    &- 2 \isd\int_{\mathbb{R}^d} (v \cdot n) f(v) \int_{\mathbb{R}^d} (v_* \cdot n) \M(v_*) |v-v_*|^\gamma \d v_* \d v \d n
    \\
    &+\isd \int_{\mathbb{R}^d} f(v) \int_{\mathbb{R}^d} (v_* \cdot n)^2 \M(v_*) |v-v_*|^\gamma \d v_* \d v\d n
    \\
    \leq
    & - \gamma_b \int_{\mathbb{R}^d} |v|^2 f(v) \int_{\mathbb{R}^d} \M(v_*) |v-v_*|^\gamma \d v_* \d v
    \\
    &+ (2 + \gamma_b) \int_{\mathbb{R}^d} |v| f(v) \int_{\mathbb{R}^d} |v_*| \M(v_*) |v-v_*|^\gamma \d v_* \d v
    \\
    &+ \int_{\mathbb{R}^d} f(v) \int_{\mathbb{R}^d} |v_*|^2 \M(v_*) |v-v_*|^\gamma \d v_* \d v.
  \end{align*}
  We can now use the following bound, which holds for all $k \geq 0$
  and some constants $0 < A_k \leq C_k$ depending on $k$:
  \begin{equation*}
    A_k (1 + |v|^\gamma)
    \leq
    \int_{\mathbb{R}^d} |v_*|^k \M(v_*) |v-v_*|^\gamma \d v_*
    \leq
    C_k (1 + |v|^\gamma),
    \qquad v \in \R^d.
  \end{equation*}
  Choosing $\epsilon > 0$ we get
  \begin{align*}
    \int_{\mathbb{R}^d} \mathcal{L}(f) |v|^2 \d v
    \leq
    & - A_0 \gamma_b \int_{\mathbb{R}^d} |v|^2 (1 + |v|^\gamma) f(v) \d v    + C_1 (2 + \gamma_b) \int_{\mathbb{R}^d} |v| (1 + |v|^\gamma) f(v) \d v
    \\
    &+ C_2 \int_{\mathbb{R}^d} f(v) (1 + |v|^\gamma) \d v \\
\leq &  \int_{\mathbb{R}^d} f(v) (C_2 + C_1(1+\gamma_b/2)/\epsilon)\left( 1 + |v|^\gamma\right) \mathrm{d}v \\
& - (A_0 \gamma_b - \epsilon C_1(1+ \gamma_b/2))\int_{\mathbb{R}^d} |v|^2(1+ |v|^\gamma) f(v) \mathrm{d}v \\
\leq &  \ird  \left( C_2 + C_1 (1+ \gamma_b/2)/\epsilon + (\epsilon C_1 (1+ \gamma_b/2) - A_0 \gamma_b)|v|^2\right)(1+|v|^\gamma) f(v) \mathrm{d}v \\
& - (A_0 \gamma_b - \epsilon C_1(1+ \gamma_b/2)) \ird |v|^2 f(v) \mathrm{d}v \\
\leq & \alpha_1 \ird f(v) \mathrm{d}v - \alpha_2 \ird |v|^2 f(v) \mathrm{d}v.
  \end{align*}
Here we choose $\epsilon$ sufficiently small to make the constant in front of the second moment negative. This also means that 
\[ (C_2 + C_1(1+ \gamma_b/2)/\epsilon + (\epsilon C_1(1+ \gamma_b/2) -
  A_0 \gamma_b ))|v|^2(1+ |v|^\gamma) \] is bounded above. These things
together give that
\begin{equation*}
  \int_{\R^d} \int_{\R^d} (\LL(f) + \TT(f)) |v|^2 \d x \d v
  \leq
  - \alpha_2 \int_{\R^d} \int_{\R^d} f |v|^2  \d x \d v
  + \alpha_1 \int_{\R^d} \int_{\R^d} f  \d x \d v,
\end{equation*}
which finishes the proof.
\end{proof}

\begin{proof}[Proof of Theorem \ref{thm:main-torus} in the case of the linear Boltzmann equation]
  We have the Doeblin condition from Lemma \ref{lem:HL2} and the
  Lyapunov structure from Lemma
  \ref{lem:HT-Boltzmann}. Harris's Theorem gives the result.
\end{proof}

\subsection{On the whole space with a confining potential}

We now work on the whole space with a confining potential. As we
stated earlier, we cannot verify the Lyapunov condition in the hard
spheres case. However, the proof for the Doeblin's condition is the
same in the hard sphere or Maxwell molecule case. We need to combine
the Lemmas \ref{lem:HTshooting}, \ref{lem:reducing} and \ref{lem:L-Boltzmann}.

\begin{lem}
  \label{lem:LBEwholespaceminorisation}
  Let the potential $\Phi \: \R^d \to \R$ be a $\mathcal{C}^2$
  function with compact level sets. Given $t > 0$ and $K > 0$ there
  exist constants $\alpha, \delta_X, \delta_V > 0$ such that for any
  $(x_0,v_0)$ with $|x_0|, |v_0| < K$ the solution $f$ to
  \eqref{LBEharris3} with initial data $\delta_{(x_0, v_0)}$ satisfies
  \[ f_t \geq \alpha \1_{ \{ |x| \leq \delta_X \}} \1_{ \{ |v| \leq \delta_V \}}. \]
\end{lem}

\begin{proof}
  We fix $R > 0$ as in Lemma \ref{lem:LBEdoeblin}. We use Lemma
  \ref{lem:HTshooting} to find constants $\alpha, R_2, R' > 0$ and an
  interval $[a,b] \subseteq (0,t)$ such that
  \begin{equation*}
    \int_{B(R')} T_s (\delta_{x_0} \1_{\{|v| \leq R_2 \}} ) \d v
    \geq
    \alpha \1_{\{|x| \leq R\}},
  \end{equation*}
  for any $x_0$ with $|x_0| \leq R$ and any $s \in [a,b]$. From Lemma
  \ref{lem:L-Boltzmann} we will use that there exists a constant
  $\alpha_L > 0$ such that
  \begin{equation}
    \label{eq:HL2}
    \mathcal{L}^+ g(x,v) \geq
    \alpha_L \left( \int_{R_L} g(x,u) \d u \right) \1_{\{|v| \leq R_2\}}
  \end{equation}
  for all nonnegative measures $g$, where $R_L := \max\{R, R'\}$. From
  Lemma \ref{lem:reducing} we can find $C_1 > 0$ (depending on $R$) such
  that
  \begin{equation*}
    f(t,x,v) \geq e^{-t C_1} \int_0^t
    \int_0^s T_{t-s} \widetilde{\mathcal{L}}^+ T_{s-r} \widetilde{\mathcal{L}}^+
    T_r (\1_{E} \delta_{(x_0,v_0)})   \d r \d s,
  \end{equation*}
  where $E$ is the set of points with energy less than $E_0$, with
  \begin{equation*}
    E_0 := \max\{ H(x,v) \,:\, |x| \leq R, |v| \leq  \max \{R_L, R_2\}  \},
  \end{equation*}
  and we recall that $\widetilde{\mathcal{L}}^+ f := \1_E \mathcal{L}^+ f$. These three
  estimates allow us to carry out a proof which is completely
  analogous to that of Lemma \ref{lem:LBEdoeblin}; notice that the
  only difference is the appearance of $R'$ here, and the need to use
  $\widetilde{\mathcal{L}}^+$ (which still satisfies a bound of the same type).
\end{proof}

Now we need to find a Lyapunov functional. As before we will look at $V$ of the form
\[ V(x,v) = \Phi(x) + \frac{1}{2}|v|^2 + \alpha x \cdot v + \beta |x|^2. \]
for some $\alpha, \beta>0$. In this case we need a stronger bound for $\Phi$, as stated in the
following:
\begin{lem} \label{lem:LBEwholespacelyapunov}
  Assume that for some $\gamma_1, \gamma_2 >0$ and some $A \in \R$ we have,
  \[ x \cdot \nabla_x \Phi(x) \geq \gamma_1 \langle x \rangle^{\gamma+2} + \gamma_2 \Phi(x) -A,\]
  where $\gamma$ is the exponent in \eqref{eq:B-splits}.
  We can find $\alpha, \beta$ such that
  \[ V(x,v) = \Phi(x) + \frac{1}{2}|v|^2 + \alpha x \cdot v +
    \beta|x|^2 \] is a function for which the semigroup associated to
  equation \eqref{LBEharris3} satisfies Hypothesis \ref{confinement}.
\end{lem}

\begin{proof}
  We are going to show that, for an appropriate choice of
  $\alpha, \beta$ it holds that
  \begin{equation*}
    (\TT^* + \LL^*)(V) \leq -\lambda V + K,
  \end{equation*}
  for some $\lambda, K > 0$, where $\LL^*$ is the dual of the
  Boltzmann collision operator, and
  $\TT^* f =  v\cdot \nabla_xf - \nabla_x \Phi \cdot \nabla_vf$ (the
  dual of the transport operator, which has the same expression as
  -$\TT$). This will show Hypotheses \ref{confinement} (see Remark
  \ref{rem:Lyapunov-equivalent}).
  
  Let us look at how the collision operator acts on the different
  terms. First,
\[ \int_{\mathbb{R}^d} \mathcal{L}(f) |v|^2 \mathrm{d}v
  =\int_{\mathbb{R}^d} \int_{\mathbb{R}^d} \int_{\mathbb{S}^{d-1}}
  f(v) \mathcal{M}(v_*) \tilde{b}(|\xi|) |v-v_*|^\gamma \left(|v'|^2 -
    |v|^2 \right) \mathrm{d}n \mathrm{d}v \mathrm{d}v_*. \] Repeating
the same calculation as in the proof of Lemma \ref{lem:HT-Boltzmann}
in this case, we see that
\[ \int_{\mathbb{R}^d}\mathcal{L}(f)|v|^2 \mathrm{d}v \leq -\alpha_1
  \int_{\mathbb{R}^d}\langle v \rangle^{\gamma+2} f(v)\mathrm{d}v +
  \alpha_2 \int_{\mathbb{R}^d} f(v) \mathrm{d}v.\]
That is,
\begin{equation}
  \label{eq:L1}
  \LL^* (|v|^2) \leq -\alpha_1 \langle v \rangle^{\gamma + 2} + \alpha_2.
\end{equation}
Similarly we have
\[ \int_{\mathbb{R}^d} \mathcal{L}(f) x \cdot v \mathrm{d}v = \int_{\mathbb{R}^d} \int_{\mathbb{R}^d} \int_{\mathbb{S}^{d-1}} f(v) \mathcal{M}(v_*) \tilde{b}(|\xi|)|v-v_*|^\gamma \left( v' \cdot x - v \cdot x\right) \mathrm{d}n \mathrm{d}v \mathrm{d}v_*. \]
We can see that
\[ v' \cdot x - v \cdot x = (v \cdot n)(x \cdot n) - (v_* \cdot n) (x \cdot n). \] Integrating this gives that
\begin{align*}
\int_{\mathbb{R}^d} \mathcal{L}(f) x\cdot v \mathrm{d}v = & \int_{\mathbb{R}^d} \int_{\mathbb{R}^d} \int_{\mathbb{S}^{d-1}} f(v) \mathcal{M}(v_*) \tilde{b}(|\xi|)|v-v_*|^\gamma (v \cdot n) (x \cdot n) \mathrm{d}v_* \mathrm{d}v \mathrm{d}n\\
\leq & \int_{\mathbb{R}^d} f(v) \langle v\rangle^{\gamma+1} |x| \mathrm{d}v,
\end{align*}
so
\begin{equation}
  \label{eq:L2}
  \LL^*(x\cdot v) \leq \langle  v\rangle^{\gamma + 1} |x|.
\end{equation}
For the effect of $\TT^*$, notice that
\[ V(x,v) = \Phi(x) + |v|^2/2 + \alpha x\cdot v + \beta |x|^2 =
  H(x,v) + \alpha x\cdot v + \beta |x|^2, \] where $H(x,v)$ denotes
the energy. We have
\begin{equation*} 
  \TT^* (H(x,v)) = 0,
  \qquad
  \TT^* (|x|^2) = 2 x \cdot v,
  \qquad
  \TT^* (x \cdot v) = |v|^2 - x \cdot \nabla_x \Phi(x)
\end{equation*}
Using this together with \eqref{eq:L1} and \eqref{eq:L2} we have
\begin{multline*}
  (\LL^* + \TT^*) (V)
  \leq
  - \frac{\alpha_1}{2} \langle v\rangle^{\gamma+2}
  + \frac{\alpha_2}{2} + \alpha \langle v\rangle^{\gamma+1}|x|
  + \alpha |v|^2
  - \alpha x \cdot \nabla_x \Phi(x)
  + 2\beta x \cdot v
  \\
  \leq
   \left (\alpha-\frac{\alpha_1}{2} \right ) \langle v\rangle^{\gamma+2} +
    (\alpha+2\beta)|x|\langle v\rangle^{\gamma+1}
  - \alpha \gamma_1 \langle x\rangle^{\gamma+2}
    - \alpha \gamma_2 \Phi(x)+ \frac{\alpha_2}{2}
    + \alpha A.
  \end{multline*}
  Setting $\beta=\alpha$, $\alpha \leq \alpha_1/4$, and using Young's
  inequality $AB \leq \frac{1}{p} A^p + \frac{1}{q} B^q$ on
  $\left( |x|/\epsilon\right)\left(\langle v \rangle^{\gamma+1}
    \epsilon \right)$ with $p=(\gamma +2)/(\gamma+1)$ and
  $q=\gamma+2$, we get
  \begin{multline*} (\LL^* + \TT^*) (V) \leq
    \left(-\frac{\alpha_1}{4}
      + 3 \epsilon^{\frac{\gamma+2}{\gamma+1}}
      \frac{\gamma+1}{\gamma+2}\right)\langle v\rangle^{\gamma +2} +
    \left(\frac{3 \alpha^{\gamma+2}}{(\gamma+2)} \frac{1}{\epsilon^{\gamma+2}} - \alpha
      \gamma_1 \right) \langle x\rangle^{\gamma+2}
    \\
    - \alpha
      \gamma_2 \Phi(x) +\frac{\alpha_2}{2} + \alpha A.
\end{multline*}
Now we can choose $\epsilon$ small enough so that the
$\langle v \rangle^{\gamma+2}$ term is negative and then for this
$\epsilon$ choose $\alpha$ small enough so that the
$\langle x \rangle^{\gamma+2}$ term is negative (since
$\gamma+2 \geq 1$). Then, since $\langle z \rangle^{\gamma+2}$ grows faster than $|z|^2$ at infinity this gives
\[ (\LL^* + \TT^*)(V)
  \leq
  -\lambda_1
  (|x|^2+|v|^2)- \lambda_2 \Phi(x) + K.
\]
Then using equivalence between the quadratic forms
\[ |x|^2 + |v|^2 \quad \mbox{and} \quad \frac{1}{2}|v|^2 + \alpha x \cdot v + \alpha |x|^2, \] when $\alpha < 1/2$ we have the result in the Lemma.
\end{proof}

\begin{proof}[Proof of Theorem \ref{thm:main-confining} in the case of the linear Boltzmann equation]
We have the minorisation condition in Lemma \ref{lem:LBEwholespaceminorisation} and the Lyapunov condition from Lemma \ref{lem:LBEwholespacelyapunov}. Therefore we can apply Harris's Theorem.
\end{proof}

\subsection{Subgeometric convergence}

As with the linear relaxation Boltzmann equation, the minorisation
results in Lemma \ref{lem:LBEwholespaceminorisation} holds for $\Phi$
which are not sufficiently confining to prove the Lyapunov
structure. However in this situation we can still prove subgeometric
rates of convergence. Here in order to find a Lyapunov functional we
need to be more precise about how $\mathcal{L}$ acts on the $x \cdot v$ moment.

We need $\Phi(x)$ to provide a stronger confinement if we are in the hard potential case. We want 
\begin{equation} \label{eq:Phi-subgeometric}
x \cdot \nabla_x \Phi(x) \geq \gamma_1 \langle x \rangle^{1+ \delta} + \gamma_2 \Phi(x) -A, \quad \Phi(x) \le \gamma_3 \langle x \rangle^{1+\delta}
\end{equation}
for some $\gamma_1, \gamma_2, \gamma_3, \delta>0$. Then we have
\begin{lem} \label{lem:LBEwholespacesubgeometriclyapunov}
Assume that $\Phi$ is a $C^2$ potential satisfying \eqref{eq:Phi-subgeometric}.
Then there exist some $\alpha, \beta >0$ satisfying $4 \alpha^2 < \beta$ such that the function 
\[ V(x,v) = \Phi(x) + \frac{1}{2}|v|^2 + \frac {\alpha x \cdot v } {\langle x \rangle} + \beta \langle x \rangle \] satisfies
\[ \UU(V) \leq -\lambda V^{\delta / (1+\delta)} + K, \] for some positive constants $\lambda,K$.
\end{lem}
\begin{rem}
	Notice that 
	\[ V(x,v) \geq \Phi(x) + \frac{1}{4} |v|^2 + (\beta -4\alpha^2) \langle x \rangle, \]
	so that the sub level sets of $V$ are bounded.
\end{rem}

\begin{proof}
  Using \eqref{eq:L1}, \eqref{eq:L2} and that 
  \begin{equation*}
   \TT^* (H(x,v)) = 0,
  \qquad
  \TT^* (\langle x\rangle) = \frac{x \cdot v}{\langle x\rangle},  \qquad  \TT^* \left (\frac{x \cdot v}{\langle x\rangle} \right )  = \frac{|v|^2}{\langle x\rangle} - \frac{(x \cdot v)^2}{\langle x\rangle ^2} - \nabla_x \Phi \cdot \frac{x}{\langle x\rangle},
  \end{equation*}
  we have the following 
  \begin{multline*}
    (\LL^* + \TT^*)(V(x,v))
    \leq
     - \alpha_1\langle v\rangle^{\gamma+2} + \alpha_2
      + \alpha \langle v\rangle^{\gamma+1} + \alpha |v|^2
     - \frac {\alpha x \cdot \nabla_x \Phi(x)}
      {\langle x \rangle}
      + \frac {\beta x \cdot v  } {\langle x \rangle}.
\end{multline*} 
  Using \eqref{eq:Phi-subgeometric}, $ \dfrac{x \cdot v}{\langle x \rangle} \leq \langle v \rangle \leq \langle v \rangle ^{2 +\gamma}$ and $\Phi(x)^{\delta / (1+\delta)} \leq \gamma_3^{\delta / (1+\delta)} \langle x \rangle ^{\delta} $ we obtain
    \begin{multline*}  
     (\LL^* + \TT^*)(V(x,v))
    \\
    \leq  (2\alpha-\alpha_1 + \beta) \langle v\rangle^{\gamma+2}
      - \alpha \gamma_1 \langle x\rangle^{\delta}
      - \alpha \gamma_2 \frac {\Phi(x)} {\langle x \rangle}
      + \alpha_2 + \alpha A
    \\
    \leq 
           \lambda_1 \left( -|v|^2  -\langle x \rangle^\delta  -
           \Phi(x)^{\delta / (1+\delta)} +C  \right)
         \leq
         \lambda V(x,v)^{\delta / (1+\delta)} 
        + K,
\end{multline*} for some $\lambda_1, K >0$.
To make the last two inequalities valid we choose $\alpha$ and $\beta$ so that $\alpha_1 > 2\alpha +\beta$ and $4 \alpha^2 < \beta$ so that 
\[ V(x,v) \geq \Phi(x) + \frac{1}{4} |v|^2 + (\beta -4\alpha^2) \langle x \rangle.\]
\end{proof}

\begin{proof}[Proof of Theorem \ref{thm:main-subgeometric} in the linear Boltzmann case]
We have the minorisation condition in Lemma \ref{lem:LBEwholespaceminorisation} and the Lyapunov condition from Lemma \ref{lem:LBEwholespacesubgeometriclyapunov}. Therefore we can apply Harris's Theorem.
\end{proof}

\section*{Acknowledgements}

The authors would like to thank to S. Mischler, C. Mouhot,
J. F\'{e}joz and R.~Ortega for some useful discussion and ideas for
some parts in the paper.

JAC and HY were supported by projects MTM2014-52056-P
and MTM2017-85067-P, funded by the Spanish government and the European
Regional Development Fund. CC was supported by grants from R\'{e}gion
Ile-de-France, the DIM PhD fellowship program. HY was also supported by
the Basque Government through the BERC 2018-2021 program and by
Spanish Ministry of Economy and Competitiveness MINECO: BCAM Severo
Ochoa excellence accreditation SEV-2017-0718 and by ``la Caixa''
Foundation. JE was supported by the UK Engineering and Physical
Sciences Research Council (EPSRC) grant EP/L016516/1 for the
University of Cambridge Centre for Doctoral Training, the Cambridge
Centre for Analysis, the European Research Council (ERC) under the
grant MAFRAN, and FSPM postdoctoral fellowship (since October 2018)
and the grant ANR-17-CE40-0030.

\bibliographystyle{plain}
\bibliography{harris}{}

\end{document}